\documentclass[11pt]{amsart}
\usepackage{amssymb}
\usepackage{amsfonts}
\usepackage{amsmath}
\usepackage{graphicx}
\usepackage{xcolor}
\usepackage{amsmath}
\usepackage{mathrsfs}
\usepackage{stmaryrd}
\usepackage{epsfig,color}
\usepackage{blindtext}
\usepackage{enumerate}
\usepackage{hyperref}
\usepackage{url}
\usepackage{bbm}
\usepackage{nicefrac,mathtools}
\usepackage{bm}   
\usepackage{tabularx}
\DeclareGraphicsExtensions{.pdf,.jpeg,.png}
\usepackage{epstopdf}
\usepackage{cancel} 
\usepackage[normalem]{ulem} 
\usepackage{verbatim} 
\usepackage{scalerel}
\usepackage{enumitem}

\usepackage{cases}

\usepackage{subcaption}

\usepackage{color}
\usepackage[msc-links, lite]{amsrefs}
\usepackage{geometry}
\usepackage{stackengine,scalerel}

\newtheorem{theorem}{Theorem}[section]

\newtheorem{proposition}[theorem]{Proposition}
\newtheorem{lemma}[theorem]{Lemma}
\newtheorem{corollary}[theorem]{Corollary}

\newtheorem{claim}[]{Claim}
\newtheorem*{acknowledgements}{Acknowledgements}
\theoremstyle{definition}
\newtheorem{definition}[theorem]{Definition}
\theoremstyle{remark}
\newtheorem{remark}[theorem]{Remark}

\numberwithin{equation}{section}

\newcommand{\mf}{\mathbf}
\newcommand{\mb}{\mathbb}
\newcommand{\mc}{\mathcal}

\newcommand{\mk}{\mathfrak}

\newcommand{\wti}{\widetilde}

\newcommand{\Vol}{\mathrm{Vol}}
\newcommand{\vol}{\mathrm{vol}}
\newcommand{\Area}{\mathrm{Area}}

\newcommand{\bd}{\partial}

\newcommand{\rom}[1]{\expandafter\romannumeral #1}
\newcommand{\Rom}[1]{\uppercase\expandafter{\romannumeral #1}}

\DeclareMathOperator{\Index}{index}
\DeclareMathOperator{\Null}{nullity}
\DeclareMathOperator{\Ker}{ker}

\DeclareMathOperator{\Ric}{Ric}

\DeclareMathOperator{\spt}{spt}

\DeclareMathOperator{\closure}{Clos}

\DeclareMathOperator{\Graph}{Graph}

\newcommand{\wcS}{\widetilde{\mathcal{S}}}

\title{Rigidity of projective area widths and a weighted Crofton formula on $RP^3$}

\author{Tongrui Wang}
\address{School of Mathematical Sciences, Shanghai Jiao Tong University, Minhang District, Shanghai 200240, China}
\email{wangtongrui@sjtu.edu.cn}

\begin{document}
\maketitle

\begin{abstract}
	We show that a Riemannian metric on $RP^3$ is a surface Zoll metric if and only if its four projective area widths are equal. 
	We also establish a weighted Crofton formula on $RP^3$ under surface Zoll metrics, from which we derive an inequality relating the systole, the volume, and the common value of the projective area widths. 
\end{abstract}

\section{Introduction}

In a Riemannian $2$-sphere $(S^2, g_{S^2})$, the {\em simple length widths} are defined as three positive min-max values $\{ \sigma_i(S^2,g_{S^2}) \}_{i=1}^3$ of the length functional. 
By the Lusternik-Schnirelmann theory \cite{lusternic1947topoligical} (see also Grayson \cite{grayson1989shortening}), each width is represented by a closed geodesic, and there are at least three distinct simple closed geodesics on $S^2$. 
Conversely, these widths determine the geometry of $S^2$ in a certain sense. 
For instance, Mazzucchelli and Suhr \cite{mazzucchelli2018characterization} confirmed a claim of Lusternik that the three simple length widths are all equal, i.e. $\{ \sigma_i(S^2,g_{S^2}) \}_{i=1}^3=\{ l \}$ if and only if all geodesics on $(S^2,g_{S^2})$ are simple closed and of equal length $l$. 
In particular, such a metric $g_{S^2}$ on $S^2$ is known as the {\em Zoll metric}, which was first constructed by Zoll \cite{zoll1903ueber} (see also \cite{guillemin1976radon}\cite{besse1978manifolds}).

In dimension $3$, one can define the {\em spherical area widths $\{\sigma_k(S^3,g_{S^3})\}_{k=1}^4$} of $(S^3,g_{S^3})$ by 
\[\sigma_k(S^3,g_{S^3}) :=  \inf _{\Phi \in \mc P_{k}'} \sup _{\Sigma \in \Phi} \Area(\Sigma), \qquad1\leq k\leq 4, \]
where $\mc P_k'$ is the set of $k$-sweepouts formed by embedded two-spheres that are continuous in the smooth topology (cf. \cite{ambrozio2024rigidity}*{\S 2}). 
It follows from the Simon-Smith min-max theory that each $\sigma_k(S^3,g_{S^3})$ is realized by the area of a disjoint union of some embedded minimal $2$-spheres. 
In \cite{ambrozio2024rigidity}, Ambrozio, Marques, and Neves also showed that the equality of the four spherical area widths characterizes those metrics on $S^3$ for which there exists a {\em Zoll family $\{\Sigma_a\}_{a\in \mb {RP}^3}$ of minimal $2$-spheres}. 
Namely, for any $p\in S^3$ and $2$-dimensional subspace $P\subset T_pS^3$, there is a unique $a\in \mb {RP}^3$ with $p\in\Sigma_a$ and $T_p\Sigma_a=P$. 
Moreover, in \cite{ambrozio2025metrics}, they also constructed metrics on higher-dimensional $S^{n+1}$ that contain Zoll families of minimal hyperspheres.

In the spirit of these results, the present paper investigates the {\em projective area widths} of a Riemannian $(RP^3, g_{RP^3})$ defined by 
\begin{align}\label{Eq: projective area widths}
	\sigma_0(RP^3,g_{RP^3}) :=  \inf _{\Sigma\in \mc S} \Area(\Sigma), \quad {\rm and}\quad \sigma_k(RP^3,g_{RP^3}) :=  \inf _{\Phi \in \mc P_{k}} \sup _{\Sigma \in \Phi} \Area(\Sigma),
\end{align}
where $1\leq k\leq 3$, $\mc S$ is the set of all embedded $RP^2$ in $RP^3$, and $\mc P_k$ is the set of $k$-sweepouts formed by $\Sigma\in\mc S$.  
Our first main result is to characterize the Riemannian metrics on $RP^3$ for which the four projective area widths are equal.

\begin{theorem}\label{main Thm: zoll metric on RP3}
	Let $(RP^3,g_{RP^3})$ be a Riemannian three-dimensional projective space. 
	Then, 
	\[ \sigma_0(RP^3,g_{RP^3}) = \sigma_1(RP^3,g_{RP^3}) = \sigma_2(RP^3,g_{RP^3})= \sigma_3(RP^3,g_{RP^3})\]
	if and only if $g_{RP^3}$ is a {\em surface Zoll metric} on $RP^3$ in the sense that there is a family $\{\Sigma_a\}_{a\in \mb {RP}^3}$ of embedded minimal projective planes in $(RP^3, g_{RP^3})$ so that for some $\alpha\in (0,1)$, 
	\begin{itemize}
		\item[(i)] $a\in \mb {RP}^3 \mapsto \Sigma_a\subset RP^3$ is a $C^1$ map into the space of $C^{3,\alpha}$ embeddings of $RP^2$ into $RP^3$;
		\item[(ii)] for any $p\in RP^3$ and $2$-dimensional subspace $P\subset T_p(RP^3)$, there is a unique $a\in \mb {RP}^3$ satisfying $p\in\Sigma_a$ and $T_p\Sigma_a=P$. 
	\end{itemize}
\end{theorem}

\begin{remark}
	We call the family $\{\Sigma_a\}_{a\in \mb {RP}^3}\subset \mc S$ satisfying (i) and (ii) a {\em Zoll family} of projective planes. 
	The existence of a nontrivial surface Zoll metric on $RP^3$ (also on $RP^{n}$ for $n\geq 4$) has been shown by Ambrozio and Guajardo in \cite{ambrozio2025equivariant}. 
	We also mention that by \cite{ambrozio2024rigidity}*{Theorem C}, the constant curvature metric on $RP^3$ is determined by its systole and $\sigma_2(RP^3,g_{RP^3})$. 
\end{remark}

The proof of Theorem \ref{main Thm: zoll metric on RP3} is based on the method in \cite{ambrozio2024rigidity} and the equivariant constructions in \cite{li2024lowgenus} for minimal $RP^2\subset RP^3$.  
As a key observation, we show that for the oriented double cover $\Sigma$ of a minimal $RP^2\subset RP^3$, every second eigenfunction $\phi_2$ of the Jacobi operator $L_\Sigma$ changes sign under the antipodal map $g_-$ on $\Sigma$, i.e. $\phi_2\circ g_-=-\phi_2$. 
In particular, for a degenerate stable minimal $RP^2\subset RP^3$, its oriented double cover $\Sigma$ has Morse index $1$ and nullity at most $3$, which generalizes a result of Cheng \cite{cheng1976eigenfunctions}. 

Note that the parameterization of the Zoll family $\{\Sigma_a\}_{a\in\mb{RP}^3}$ is not unique (Remark \ref{Rem: isomorphism of Jacobi map}). 
Nevertheless, the proof of Theorem \ref{main Thm: zoll metric on RP3} yields a {\em nondegenerate} parameterization so that the map $a\in\mb{RP}^3\mapsto \Sigma_a\subset RP^3$ is a $C^1$-embedding into the space of $C^{3,\alpha}$ embeddings of $RP^2$ into $RP^3$. 
As an application, we establish a weighted Crofton formula on $RP^3$ under surface Zoll metrics. 

\begin{theorem}\label{Main Thm: weighted Crofton formula}
	Let $g_{RP^3}$ be a surface Zoll metric on $RP^3$, and let $\{\Sigma_a\}_{a\in \mb {RP}^3}$ be the family of embedded minimal projective planes satisfying Theorem \ref{main Thm: zoll metric on RP3} (i) and (ii) with a nondegenerate parameterization. 
	Then, $\mc I:=\{(a,p)\in \mb {RP}^3\times RP^3: p\in \Sigma_a\}$ is a compact $C^1$-embedded hypersurface in $\mb {RP}^3\times RP^3$. 
	Additionally, there exist a continuous function $J:\mc I\to(0,\infty)$ and a probability measure $\mu_A$ on the parameter space $A:=\mb {RP}^3$ such that
	\begin{itemize}
		\item[(i)] $\int_{\Sigma_a} J(a,p) d\vol_{\Sigma_a}(p)=W$, where $W$ is the common value of the projective area widths of $(RP^3,g_{RP^3})$;
		\item[(ii)] for any $f\in C^0(RP^3)$, 
		\[\int_A\int_{\Sigma_a} f(p) J(a,p) d\vol_{\Sigma_a}(p)d\mu_A(a) = \frac{W}{\Vol(RP^3,g_{RP^3})} \int_{RP^3} f d\vol_{RP^3}  ;\]
		\item[(iii)] for any $C^1$-embedded compact curve $\gamma\subset RP^3$, 
		\[ \int_A \Big( \sum_{p\in \gamma\cap \Sigma_a} J(a,p)  \Big) ~d\mu_A(a) = \frac{W}{2\cdot\Vol(RP^3,g_{RP^3})} L_{g_{RP^3}}(\gamma) ,\]
	\end{itemize}
	where $d\vol_{\Sigma_a}$ and $d\vol_{RP^3}$ are the volume densities induced by $g_{RP^3}$ on $\Sigma_a$ and $RP^3$ respectively, and $L_{g_{{RP}^3}}(\gamma)$ is the length of $\gamma$ under the metric $g_{RP^3}$. 
\end{theorem}

We mention that the proof of Theorem \ref{Main Thm: weighted Crofton formula}(iii) can also be generalized to rectifiable curves in $RP^3$. 
Additionally, the weight function $J$ and the probability measure $\mu_A$ in Theorem \ref{Main Thm: weighted Crofton formula} depend only on the surface Zoll metric $g_{RP^3}$ and the nondegenerate parameterization $a\in \mb RP^3\mapsto \Sigma_a\subset RP^3$ of the Zoll family (Remark \ref{Rem: intrinsic and rectifiable curves}). 
We refer to \cite{ambrozio2026spheres}*{\S 5.2} for a related problem on the existence of a certain probability measure on the parameter space. 
Independently, Martins \cite{martins2026spectral} obtained the same rigidity characterization (Theorem \ref{main Thm: zoll metric on RP3}) by adapting the deformation scheme of Ambrozio-Marques-Neves \cite{ambrozio2024rigidity} in the space of embedded projective planes and using the topology of transitive families; our paper organizes the proof through $G_\pm$-equivariant lifts to $S^3$ and additionally establishes the weighted Crofton formula, which is not addressed there.

Combining the weighted Crofton formula with the result of Ambrozio-Marques-Neves \cite{ambrozio2024rigidity}*{Theorem C}, we have the following inequalities that relate the systole, the volume, and the common value of the projective area widths on $RP^3$ under a surface Zoll metric. 
Recall that the systole $\text{sys}(RP^3, g_{RP^3})$ is defined as the least length of a non-contractible loop in $(RP^3, g_{RP^3})$.
\begin{corollary}
    Let $g_{RP^3}$ be a surface Zoll metric on $RP^3$. 
    Using the notations in Theorem \ref{Main Thm: weighted Crofton formula}, let $j(a):=\min_{p\in\Sigma_a}J(a,p)$ for any $a\in A=\mb {RP}^3$, and let $\varsigma:=\int_{A}j(a)d\mu_A\in (0,1]$. 
    Then, 
    \[ \varsigma\cdot\frac{2\cdot\Vol(RP^3,g_{RP^3})}{W}\leq \text{sys}(RP^3,g_{RP^3}) \leq \sqrt{\frac{\pi W}{2}}. \]
    In particular, $8\varsigma^2\cdot(\Vol(RP^3,g_{RP^3}))^2\leq \pi W^3$ with equality if and only if $g_{RP^3}$ has constant sectional curvature. 
\end{corollary}

The investigation of area widths is motivated by the inverse problem, which asks whether these invariants also encode geometric information about the ambient space.  
Research into such problems is of long-standing interest and significance. 
For instance, the celebrated isospectral problem \cite{kac1966can} is to recover the geometric properties of the underlying space from its spectrum of the Laplacian. 
In analogy with the spectrum of the Laplacian, Gromov \cite{gromov1988dimension} proposed the notion of the {\em volume spectrum} $\{\omega_k(M^{n+1},g_{M})\}_{k\in\mb N}$ for closed Riemannian manifolds $(M^{n+1}, g_{M})$ as the min-max values of the area functional in the space $\mc Z_n(M;\mb Z_2)$ of mod $2$ $n$-cycles (see also \cite{marques2017existence} for specific definitions). 
Subsequent research has revealed that the volume spectrum shares many properties analogous to the spectrum of the Laplacian, e.g. the nonlinear growth rate \cite{gromov2003isoperimetry}\cite{guth2009minimax} and the Weyl asymptotic law \cite{liokumovich2018weyl}. 
Moreover, analogous to the realization of Laplacian eigenvalues by the energy of eigenfunctions, it follows from Almgren-Pitts min-max theory that every volume spectrum is realized as the area of certain minimal hypersurfaces. 
By virtue of this deep connection, the volume spectrum has played a crucial role in the study of minimal hypersurfaces, including the resolution of Yau's conjecture (\cite{marques2017existence}\cite{song2018existence}), the spatial distribution for minimal hypersurfaces (\cite{irie2018density}\cite{marques2019equidistribution}\cite{song2021generic}), and the Morse theory for the area functional (\cite{marques2016morse}\cite{marques2021morse}\cite{zhou2020multiplicity}). 
Therefore, it is natural to study to what extent the volume spectrum (and the area widths) determine the geometry of the ambient manifold.

In dimension $n+1=2$, Chodosh and Mantoulidis \cite{chodosh2023p-widths} developed a min-max theory to show that the $k$-width (i.e. the $k$-th volume spectrum) of a Riemannian surface is realized by the lengths of immersed closed geodesics. 
In particular, they proved that the unit sphere $(\mb S^2_1,g_{\mb S^2_1})$ with the round metric has its $k$-width $\omega_k(\mb S^2_1, g_{\mb S^2_1})=2\pi \lfloor \sqrt{k}\rfloor$ for all $k\geq 1$. 
Based on this result, Marx-Kuo \cite{marx-kuo2025isospectral} disproved the $k$-widths isospectral problem for $(\mb S^2_1,g_{\mb S^2_1})$ by a family of Zoll metrics on $S^2$, which admit the same volume spectrum. 
In contrast to $S^2$, Ambrozio, Marques, and Neves \cite{ambrozio2024rigidity} showed that any compact Riemannian manifold with the same volume spectrum as the round $(\mb {RP}^2,g_{\mb {RP}^2})$ is isometric to it. 

\begin{acknowledgements}
	The author thanks professor Xin Zhou for bringing the rigidity problem to his attention and for helpful discussions. 
	T.W. is supported by the National Natural Science Foundation of China 12501076, Tianyuan Mathematics Frontier Key Special Program 12526203, the Natural Science Foundation of Shanghai 25ZR1402252, and Shanghai Qi-Guang Scholarship. 
\end{acknowledgements}

\noindent{\bf AI Disclosure}
{\it 
	Sections 2 and 3 were originally completed in June 2025 without AI assistance.
    The author used ChatGPT-5.6 Sol to explore specific questions concerning the tangent map $d\mc T$ (Lemma \ref{Lem: C1 diffeomorphism}) and the pullback density $(\pi_A^\gamma)^*d\mu_A$ (Theorem \ref{Thm: weighted Crofton formula}). 
	The tool was also used to suggest improvements to the grammar and phrasing of the manuscript.
	The author wrote the entire manuscript and takes full responsibility for its content. 
}

\section{Preliminaries}

Let $(RP^3, g_{RP^3})$ be a Riemannian $RP^3$ and $\pi: S^3\to RP^3$ be the double cover. 
Denote by $g_{S^3}$ the Riemannian metric on $S^3$ so that $\pi$ is a local isometry, and by $g_-$ the antipodal map on $S^3$. 
Then for any integer $k\geq 2$ and $\alpha\in(0,1)$, we use the following notation:
\begin{itemize}
	\item $ g_{\mb {RP}^3}$ and $ g_{\mb S^3}$ are the standard round metrics on $RP^3$ and $S^3$ respectively; 
	\item $G:=\mb Z_2=\{id, g_-\}$ is the deck transformation group of $\pi: S^3\to RP^3$ acting by isometries on $(S^3, g_{S^3})$;
	\item $G_+:=\{id\}\subset G$ is the index-two subgroup, and $G_-:=\{g_-\}$ is the coset of $G_+$; 
	\item $\mc S$ is the space of smooth embedded $RP^2$ in $RP^3$ with smooth topology;
	\item $\wcS_{G_\pm}:=\{\pi^{-1}(\Sigma)\subset S^3: \Sigma\in \mc S\}$ with smooth topology; 
    \item $\wcS^{k,\alpha}_{G_\pm}:=\{\mbox{$C^{k,\alpha}$ embedded $G$-invariant $S^2$ in $S^3$ with $\pi(S^2)\cong RP^2$}\}$;
    \item $\wcS_{G}:=\{\pi^{-1}(\Gamma)\subset S^3: \mbox{$\Gamma$ is an embedded $S^2$ or $RP^2$ in $RP^3$} \}$ with smooth topology;
	\item $\mk X(RP^3), \mk X(S^3)$ are the spaces of smooth vector fields on $RP^3$ and $S^3$ respectively;
	\item $ \mk X^G(S^3):=\{X\in \mk X(S^3): dg(X) = X, \forall g\in G\}$;
    \item $C^\infty_{G_\pm}(\wti\Sigma), C^{k,\alpha}_{G_\pm}(\wti\Sigma)$ are the spaces of smooth functions and $C^{k,\alpha}$ functions on $\wti\Sigma\in \wcS_{G_\pm}^{k,\alpha}$ that change signs under $g_-$, i.e. $h\circ g_-=-h$.
\end{itemize} 
Note that $g_-$ acts by an orientation-reversing isometry on any $\wti\Sigma\in \wcS_{G_\pm}$. 
Additionally, $\wcS_{G_\pm} \subsetneq \wcS_{G}$, and $d\pi : \mk X^G(S^3) \to \mk X(RP^3)$ is a bijection.

\subsection{Variations in $RP^3$ and $S^3$}
Given any $ \Sigma\in \mc S$ and $X\in \mk X(RP^3)$, let $\varphi_t$ be the diffeomorphisms generated by $X$. Then
\[ \delta\Sigma(X):= \left . \frac{d}{dt} \right|_{t=0} \Area(\varphi_t(\Sigma)) = -\int_{\Sigma} \langle H_{\Sigma}, X\rangle , \] 
where $H_{\Sigma}$ is the mean curvature vector field of $\Sigma$. 
If $\Sigma$ is stationary, i.e. $\delta\Sigma(X)\equiv 0$ for all $ X\in \mk X(RP^3)$, then $H_{\Sigma}=0$ and $\Sigma$ is a minimal $RP^2$ in $RP^3$. 
In this case, we have 
\begin{align*}
	\delta^2\Sigma(X):= \left. \frac{d^2}{dt^2}\right|_{t=0} \Area(\varphi_t(\Sigma)) &= Q_{\Sigma}(X^\perp, X^\perp) =  - \int_{\Sigma} \langle L_{\Sigma} X^\perp, X^\perp\rangle
	\\&=\int_{\Sigma} |\nabla^\perp X^\perp |^2 - \left( |A_{\Sigma}|^2|X^\perp|^2 + \Ric_{RP^3}(X^\perp,X^\perp) \right),
\end{align*}
where $X^\perp$ is the normal component of $X$ along $\Sigma$, $L_{\Sigma}: \mk X^\perp(\Sigma)\to \mk X^\perp(\Sigma)$ is the Jacobi operator defined on the space of normal vector fields $\mk X^\perp(\Sigma)$, and $A_{\Sigma}$ is the second fundamental form of $\Sigma$. 
Note that $L_{\Sigma}$ admits a discrete spectrum 
\[\lambda_1\leq\lambda_2\leq\dots\leq\lambda_k\leq\dots.\]
Then we define 
\begin{itemize}
	\item the $k$-th eigenvector field $X_k\in \mk X^\perp(\Sigma)$: $L_{\Sigma} X_k= - \lambda_k X_k$;
 	\item $\Index(\Sigma)$: the maximal dimension of a linear subspace of $\mk X^\perp(\Sigma)$ where $Q_{\Sigma}$ is negative definite, which is the number of negative eigenvalues of $L_{\Sigma}$ (counting multiplicities);
	\item $\Null(\Sigma)$: the dimension of the subspace $\Ker(L_{\Sigma})$;
\end{itemize}
We say $\Sigma$ is stable if $\Index(\Sigma)=0$, i.e. $\lambda_1\geq 0$. 

\medskip
Next, for any $\wti\Sigma \in \wcS_{G_\pm}$, we can similarly define the first variations $\delta\wti\Sigma$, the second variations $\delta^2\wti\Sigma$, the Jacobi operator $ L_{\wti\Sigma}$, the spectrum 
\[\tilde\lambda_1 < \tilde\lambda_2\leq\dots\leq \tilde\lambda_k\leq\dots,\] 
the eigenvector fields $ \wti X_k$ of $L_{\wti\Sigma}$, $\Index( \wti\Sigma)$ and $\Null(\wti\Sigma)$. 
We also say that $\wti\Sigma$ is {\em $G$-stationary in $S^3$} if $\delta \wti\Sigma(X) =0$ for all $X\in \mk X^G(S^3)$. 
By Palais's principle of symmetric criticality, $\wti\Sigma$ is $G$-stationary if and only if $\wti\Sigma$ is minimal, i.e. $H_{\wti\Sigma}=0$ (see for instance \cite[Lemma 6]{wang2022min}). 

In addition, given any $\wti\Sigma\in\wcS_{G_\pm}$, $\wti\Sigma$ admits a global unit normal $\tilde\nu$ so that $dg_-(\tilde\nu)=-\tilde\nu$. 
Hence, by identifying a normal vector field $\wti X= f\tilde\nu$ with $f\in C^\infty(\wti\Sigma)$, we can rewrite the Jacobi operator $L_{\wti\Sigma}: C^\infty(\wti\Sigma)\to C^\infty(\wti\Sigma)$ by
\[ L_{\wti\Sigma} f= \Delta_{\wti\Sigma} f + \left(|A_{\wti\Sigma}|^2 + \Ric_{S^3}(\tilde\nu,\tilde\nu)\right) f , \]
and define $\phi_k\in C^\infty(\wti\Sigma )$ as the $k$-th eigenfunction of $L_{\wti\Sigma}$, i.e. $L_{\wti\Sigma}\phi_k=-\tilde\lambda_k\phi_k$ and $\wti X_k=\phi_k\tilde\nu$. 

Moreover, we can also restrict $L_{\wti\Sigma}$ to the space of $G$-invariant normal vector fields 
\[\mk X^{\perp,G}(\wti\Sigma):=\{\wti X\in \mk X^{\perp}(\wti\Sigma): dg_-(\wti X)=\wti X\},\] 
and obtain the discrete equivariant spectrum
\[\tilde\lambda_1^G\leq \tilde\lambda_2^G\leq \dots\leq \tilde\lambda_k^G\leq \dots.\]
Then we have the {\em $k$-th equivariant eigenvector field} $ \wti X^G_k\in \mk X^{\perp,G}(\wti\Sigma)$ with $L_{\wti\Sigma}  \wti X^G_k = - \tilde\lambda_k^G \wti X^G_k$, and the {\em $k$-th equivariant eigenfunction} $\phi_k^G:=\langle \wti X^G_k, \tilde \nu\rangle$ so that $\phi_k^G\circ g_- = - \phi_k^G$. 
Clearly, for $\Sigma=\pi(\wti\Sigma)$,
\begin{align}
	\tilde \lambda_k^G = \lambda_k \quad{\rm and}\quad d\pi( \wti X^G_k)= X_k,
\end{align}
because $\pi:S^3\to RP^3$ is a local isometric double cover. 
We also denote by $\Index_G( \wti\Sigma)$ the number of negative equivariant eigenvalues of $L_{\wti\Sigma}$, which can be similarly defined for $\wti\Sigma\in\wcS_G$. 

\medskip
We next show that the second eigenfunctions of a minimal sphere $\wti \Sigma\in\wcS_{G_\pm}$ are in one-to-one correspondence with the first eigenvector fields of the minimal projective plane $\Sigma=\pi(\wti\Sigma)\in\mc S$. 
\begin{proposition}\label{Prop: 2nd eigenfunctions}
	Let $\wti\Sigma\in\wcS_{G_\pm}$ be a minimal $2$-sphere in $(S^3,g_{S^3})$, and $\Sigma:=\pi(\wti\Sigma)\in\mc S$ be a minimal projective plane in $(RP^3,g_{RP^3})$. 
	Then, using the above notation,
	\[\tilde\lambda_1< \tilde \lambda_2=\tilde \lambda_1^G =\lambda_1,\]
	and every second eigenfunction $\phi_2$ of $\wti \Sigma$ satisfies 
	\[\phi_2(g_-\cdot x)=-\phi_2(x),\quad \forall x\in\wti\Sigma.\]
	Additionally, we have a well-defined bijection
	\[d\pi: \{ \wti X\in \mk X^\perp(\wti\Sigma): L_{\wti\Sigma}  \wti X =- \tilde \lambda_2  \wti X \} \to \{ X\in\mk X^\perp( \Sigma): L_{\Sigma} X = -\lambda_1 X \} .\]
	In particular, if $\lambda_1=0$, then $\Index(\wti\Sigma)=1$ and $1\leq \Null(\wti\Sigma)\leq 3$. 
\end{proposition}

\begin{proof}
	Firstly, note that every $ X\in \mk X^\perp(\Sigma)$ corresponds uniquely to a normal vector field $\wti X\in \mk X^{\perp,G}(\wti \Sigma)$ with $d\pi (\wti X) =  X$ so that $h=\langle \wti X,  \tilde \nu\rangle\in C^\infty(\wti\Sigma)$ satisfies $h(g_-\cdot x)=-h(x)$ for all $x\in\wti \Sigma$. 
	Hence, every first eigenvector field $ X_1$ of $\Sigma$ can be lifted as a $G$-invariant eigenvector field $\wti X$ and an eigenfunction $h\in C^\infty_{G_\pm}(\wti\Sigma)$ of $\wti \Sigma$.  
	In particular, $\lambda_1=\tilde\lambda_1^G=\tilde \lambda_k$, where $k\geq 1$ is the smallest integer so that there exists a $k$-th eigenfunction $\phi_k$ of $\wti\Sigma$ with $\phi_k\in C^\infty_{G_\pm}(\wti\Sigma)$. 
	
	Then, since the first eigenvalue $\tilde \lambda_1$ of $\wti\Sigma$ has multiplicity one and the first eigenfunction $\phi_1$ has a fixed sign, we know $\phi_1\notin C^\infty_{G_\pm}(\wti\Sigma)$, and thus $\tilde\lambda_1 < \tilde \lambda_2 \leq \lambda_1$. 
	Next, we make the following claim.
	
	\begin{claim}\label{Claim: 2nd eigenfunctions}
		Every second eigenfunction $\phi_2$ of $\wti\Sigma$ satisfies $\phi_2(g_-\cdot x)=-\phi_2(x)$ for all $ x\in\wti\Sigma$. 
	\end{claim}
	\begin{proof}[Proof of Claim \ref{Claim: 2nd eigenfunctions}]
		Suppose there exists $x_0\in\wti\Sigma$ so that $\phi_2(g_-\cdot x_0) + \phi_2(x_0) \neq 0$. 
		Then, \[\Phi_2:=\phi_2+ \phi_2\circ g_-\]
		is a non-trivial (as $\Phi_2(x_0)\neq 0$) $G$-invariant smooth function on $\wti\Sigma$, which is also a second eigenfunction of $\wti\Sigma$. 
		Hence, by Cheng \cite{cheng1976eigenfunctions}, $\Phi_2$ has two nodal domains, and $\nabla^{\wti\Sigma}\Phi_2$ does not vanish on the nodal set $\{\Phi_2=0\}$. 
		It then follows that $\{\Phi_2=0\}$ is a simple closed curve on the $2$-sphere $\wti\Sigma$, and the nodal domains 
		\[\Omega_+ :=\{x\in\wti\Sigma: \Phi_2(x)>0\} \quad{\rm and} \quad \Omega_- :=\{x\in\wti\Sigma: \Phi_2(x)<0\} \]
		are both homeomorphic to a $2$-dimensional disk $\mb D$. 
		Since $\Phi_2$ is $G$-invariant, we have $g_-\cdot \Omega_+=\Omega_+$, which implies that $g_-$ has a fixed point in $\Omega_+\cong \mb D$ by Brouwer's fixed point theorem. 
		This contradicts that $g_-$ is a deck transformation of $\pi:S^3\to RP^3$. 
	\end{proof}
	
	Therefore, $\tilde\lambda_2 = \lambda_1$, and $d\pi$ is a well defined bijection from the second eigenspace of $\wti\Sigma\in\wcS_{G_\pm}$ to the first eigenspace of $\Sigma\in\mc S$. 
	In particular, if $\lambda_1=0$, then $\tilde\lambda_1<\tilde \lambda_2=0$, which implies that $\Index(\wti \Sigma)=1$ and $1\leq \Null(\wti\Sigma)\leq 3$ by Cheng \cite{cheng1976eigenfunctions}. 
\end{proof}

The following proposition is an equivariant generalization of \cite{ambrozio2024rigidity}*{Proposition 2.3}. 

\begin{proposition}\label{Prop: nullily and graph}
	Let $\wti\Sigma\in\wcS_{G_\pm}$ be minimal in $(S^3,g_{S^3})$, $k\geq 2$ be an integer, and $\alpha\in (0,1)$. 
	Then there exists a neighborhood $\mc W\subset \ker(L_{\wti\Sigma})$ of the origin and a smooth embedding $\varphi: \mc W\to C^{k,\alpha}(\wti\Sigma)$ with a small constant $\eta_{\wti\Sigma}>0$ so that
	\begin{itemize}
		\item $\varphi(0)=0$, $d\varphi(0)=Id$;
		\item for any closed minimal surface $\wti\Sigma'$ in $(S^3,g_{S^3})$ with $\mf F(|\wti\Sigma|,|\wti\Sigma'|)\leq \eta_{\wti\Sigma}$, there exists $z\in\mc W$ satisfying 
			\begin{align}
				\wti\Sigma'=\Graph(\varphi(z)) := \exp^\perp_{\wti\Sigma}(\varphi(z) \tilde\nu), 
			\end{align}
			where $\exp^\perp_{\wti\Sigma}$ is the normal exponential map. 
	\end{itemize}
	In particular, if $0=\tilde \lambda_2=\tilde\lambda_1^G$, then we can take $\varphi(\mc W)\subset C^{k,\alpha}_{G_\pm}(\wti\Sigma)$. 
\end{proposition}

\begin{proof}
	By Allard regularity \cite{allard1972first}, if $\wti\Sigma'$ is a closed minimal surface that is sufficiently close to $\wti\Sigma$ in the varifold topology, then it is also close to $\wti\Sigma$ in the smooth topology. 
	Hence, the first part of the proposition follows from White \cite{white1991space}*{Theorem 1.3} (cf. \cite{ambrozio2024rigidity}*{Proposition 2.3}). 
	Moreover, note that \cite{white1991space}*{Theorem 1.3} was proved in general Banach spaces and can be applied to $\mk X_{k,\alpha}^\perp(\wti\Sigma)$, the space of $C^{k,\alpha}$ sections of the normal bundle $N\wti\Sigma$. 
	Hence, if $0=\tilde\lambda_2(\wti\Sigma)=\tilde\lambda_1^G(\wti\Sigma)$, then $\ker(L_{\wti\Sigma})\subset \mk X^{\perp,G}(\wti\Sigma)$, and we can apply \cite{white1991space}*{Theorem 1.3} to the Banach subspace $\mk X_{k,\alpha}^{\perp,G}(\wti\Sigma)\subset \mk X_{k,\alpha}^{\perp}(\wti\Sigma)$ of $G$-equivariant elements. 
	This shows the last statement since $h\tilde \nu\in \mk X^{\perp,G}_{k,\alpha}(\wti\Sigma)$ if and only if $h\in C^{k,\alpha}_{G_\pm}(\wti\Sigma)$. 
\end{proof}

\subsection{$G_\pm$-spherical area widths}
We now collect some notation in geometric measure theory from \cite{federer2014geometric}\cite{simon1983lectures}\cite{pitts2014existence}. 
\begin{itemize}
	\item $\mc Z_2(S^3;\mb Z_2)$: the space of $2$-dimensional mod $2$ flat chains $T=\bd U$ for some $3$-dimensional mod $2$ flat chain $U$ in $S^3$;
	\item $|T|$ and $\|T\|$: the integral varifold and the Radon measure in $S^3$ associated with $T\in \mc Z_2(S^3;\mb Z_2)$; 
	\item $\mc V_2(S^3)$: the closure of the space of $2$-dimensional rectifiable varifolds supported in $S^3$ in the varifold topology;
	\item $\mf F$: the $\mf F$-metric defined in the book of Pitts \cite{pitts2014existence}*{Page 66}, which induces the varifold weak topology on any bounded subset of $\mc V_2(S^3)$; 
	\item $\mc F,\mf M,\mf F$: the flat metric, mass norm, and the $\mf F$-metric on $\mc Z_2(S^3;\mb Z_2)$ respectively.
\end{itemize}
The space $\mc Z_2(S^3;\mb Z_2)$ is weakly homotopically equivalent to $\mb {RP}^\infty$ (by \cite{almgren1962homotopy}). Hence, 
\[ H^k(\mc Z_2(S^3;\mb Z_2); \mb Z_2) = \mb Z_2 = \{0,\bar \lambda^k\},\]
where $\bar\lambda$ is the generator of $H^1(\mc Z_2(S^3;\mb Z_2);\mb Z_2) = \mb Z_2$, and $\bar\lambda^k$ is the $k$-th cup product of $\bar\lambda$ with itself. 
Let $i:\wcS_{G_\pm} \to \mc Z_2(S^3;\mb Z_2)$ be the natural inclusion, which is continuous in the $\mf F$-topology. 

\begin{definition}
	For any integer $k\geq 1$, a {\em smooth spherical $(G_\pm, k)$-sweepout} is a continuous map $\Phi: X \to \wcS_{G_\pm}$, where $X$ is any finite-dimensional compact simplicial complex, so that
	\[ (i\circ\Phi)^*(\bar\lambda^k)\neq 0 \in H^k(X;\mb Z_2). \]
	Denote by $\mc {P}_k^{G_\pm}$ the set of smooth spherical $(G_\pm, k)$-sweepouts. 
	Additionally, we also define the {\em $G_\pm$-spherical area widths} of $(S^3, g_{S^3})$ by
	\[\sigma_0^{G_\pm} (S^3, g_{S^3}):= \inf_{\Sigma\in \wcS_{G_\pm}} \Area(\Sigma) \qquad{\rm and }\qquad \sigma_k^{G_\pm}(S^3, g_{S^3}):= \inf_{\Phi\in\mc P_k^{G_\pm}} \sup_{x\in{\rm dmn}(\Phi)} \Area (\Phi(x)), \]
	where $1\leq k\leq 3$. 
\end{definition}

Given $k_0\geq 2$ and $\alpha_0\in (0,1)$, one can use $\wcS_{G_\pm}^{k_0,\alpha_0}$ in place of $\wcS_{G_\pm}$ to similarly define the $C^{k_0,\alpha_0}$ spherical $(G_\pm, k)$-sweepout. 
For simplicity, we refer to them as spherical $(G_\pm, k)$-sweepouts. 

\begin{remark}
	After parameterizing $S^3$ by $\{(x_1,x_2,x_3,x_4)\in \mb R^4: \sum_{i=1}^4x_i^2=1\}$, the map
		\[\Phi: ([a_1:a_2:a_3:a_4])\in \mb {RP}^3 \mapsto \{x\in S^3: a_1x_1+a_2x_2+a_3x_3+a_4x_4=0\}\in \wcS_{G_\pm}\]
	is a smooth spherical $(G_\pm, 3)$-sweepout. 
	Hence, $\mc {P}_k^{G_\pm}\neq \emptyset$ for $1\leq k\leq 3$, and 
	\[ \sigma_0^{G_\pm} (S^3, g_{S^3})\leq \sigma_1^{G_\pm} (S^3, g_{S^3}) \leq \sigma_2^{G_\pm} (S^3, g_{S^3})\leq \sigma_3^{G_\pm} (S^3, g_{S^3}) \] 
	are well defined. 
	In addition, Hatcher's proof of the Smale conjecture in \cite{hatcher1983smale} indicates that the space of embedded $S^2\subset S^3$ deformation retracts onto the space of great spheres, which is homeomorphic to $\mb {RP}^3$. 
	Hence, $(i\circ\Phi)^*(\bar\lambda^4)=0$ and $\mc {P}_{k\geq 4}^{G_\pm}=\emptyset$ (see also \cite{ambrozio2024rigidity}*{Proposition 2.1}). 
\end{remark}


\begin{remark}
	Fix any $\Sigma_0\in \mc S$. Let $\mc Z_2'(RP^3;\mb Z_2)$ be the space of $2$-dimensional mod $2$ cycles $T$ in $RP^3$ so that $T=\llbracket\Sigma_0\rrbracket + \bd U$ for some $3$-dimensional mod $2$ flat chain $U$ in $RP^3$. Every such $T$ represents the non-trivial $\mb Z_2$-homology class in $RP^3$. 
	By \cite{wang2026multiplicity}*{\S 5.2}, $\mc Z_2'(RP^3;\mb Z_2)$ is also weakly homotopically equivalent to $\mb {RP}^\infty$. 
	Hence, using the inclusion $i: \mc S\to \mc Z_2'(RP^3;\mb Z_2)$, we can similarly define the {\em smooth projective $k$-sweepout} as a continuous map $\Phi: X\to \mc S$ satisfying $(i\circ \Phi)^*(\bar\lambda^k)\neq 0 \in H^k(X;\mb Z_2)$, where $\bar\lambda$ is the generator of $H^1(\mc Z_2'(RP^3;\mb Z_2);\mb Z_2)$ and $1\leq k\leq 3$. 
	Then, the {\em projective area widths} of $(RP^3, g_{RP^3})$ 
	\[\sigma_0(RP^3,g_{RP^3})\leq \sigma_1(RP^3,g_{RP^3})\leq \sigma_2(RP^3,g_{RP^3})\leq \sigma_3(RP^3,g_{RP^3})\]  
	are defined as in \eqref{Eq: projective area widths}. 
	We mention that $\sigma_k^{G_\pm}(S^3, g_{S^3})= 2\sigma_k(RP^3, g_{RP^3})$ for $0\leq k \leq 3$. 
\end{remark}

\begin{definition}
    For $1\leq k\leq 3$ and a sequence $\{\Phi_i\}_{i\in\mb N}\subset \mc P_k^{G_\pm}$, the {\em image set} of $\{\Phi_i\}_{i\in\mb N}$ is defined by 
    \[\mf\Lambda (\{\Phi_i\}_{i\in\mb N}):=\{V\in\mc V_2(S^3): \exists i_j\to\infty, x_{i_j}\in{\rm dmn}(\Phi_{i_j}) ~s.t.~V=\lim_{j\to\infty}|\Phi_{i_j}(x_{i_j})|\}.\]
    Additionally, we call $\{\Phi_i\}_{i\in\mb N}\subset \mc P_k^{G_\pm}$ a {\em min-max sequence} for $\sigma_k^{G_\pm}$ if 
    \[\limsup_{i\to\infty} \sup_{x\in{\rm dmn}(\Phi_i)} \Area(\Phi_i(x)) = \sigma_k^{G_\pm}(S^3,g_{S^3}).\]
    The {\rm critical set} of $\{\Phi_{i}\}_{i\in\mb N}$ is defined by 
    \[\mf C(\{\Phi_{i}\}_{i\in\mb N}) := \{V\in \mf \Lambda(\{\Phi_{i}\}_{i\in\mb N}): \|V\|(S^3)=\sigma_k^{G_\pm}(S^3,g_{S^3})\}.\]
\end{definition}

The following result follows from the equivariant minimizing/min-max constructions \cite{li2024lowgenus}\cite{wang2023equivariant}\cite{franz2023index}. 

\begin{theorem}\label{Thm: min-max}
	For each $0\leq k \leq 3$, there exist a disjoint family $\{\wti\Sigma^{(k)}_i\}_{i=0}^{l_k}\subset \wcS_G$, ($l_0=0$, $l_k\geq 0$), of smooth embedded minimal surfaces in $(S^3,g_{S^3})$ and $\{m_i^{(k)}\}_{i=0}^{l_k}\subset \mb Z_+$ so that
	\[\sigma_0^{G_\pm}(S^3, g_{S^3})= \Area(\wti\Sigma^{(0)}_0) \quad{\rm and}\quad \sigma_k^{G_\pm}(S^3, g_{S^3}) = \sum_{i=0}^{l_k} m_i^{(k)}\Area(\wti\Sigma_i^{(k)}), ~{\rm for~}1\leq k\leq 3.\]
	Moreover, we also have
	\begin{itemize}
		\item[(i)] $\wti\Sigma^{(k)}_0\in \wcS_{G_\pm}$, and $\{\wti\Sigma^{(k)}_i\}_{i=1}^{l_k}\subset \wcS_G\setminus \wcS_{G_\pm}$; 
		\item[(ii)] $m^{(k)}_i=1$ provided that $\wti\Sigma^{(k)}_i$ is unstable;
		\item[(iii)] $\pi(\wti\Sigma^{(0)}_0)$ is a stable minimal projective plane in $RP^3$, and thus $\tilde\lambda_2(\wti\Sigma^{(0)}_0)=\lambda_1(\pi(\wti\Sigma^{(0)}_0))\geq 0$, $\Index(\wti\Sigma^{(0)}_0)\leq 1$ (by Proposition \ref{Prop: 2nd eigenfunctions});  
		\item[(iv)] $\sum_{i=0}^{l_k}\Index_{G}(\wti\Sigma_i^{(k)}) \leq k$ for $1\leq k\leq 3$. 
	\end{itemize}
\end{theorem}
\begin{proof}
	The existence of $\{\wti\Sigma^{(k)}_i\}_{i=0}^{l_k}$ and $\{m_i^{(k)}\}_{i=0}^{l_k}\subset \mb Z_+$ together with (ii) and (iii) follows from the area minimizing construction in \cite{bray2010area} and the equivariant min-max constructions in \cite{li2024lowgenus}*{Theorem 5.3}. 
	The statement in (i) follows from \cite{li2024lowgenus}*{Corollary 6.1}. 
	By the equivariant index upper bounds \cite{wang2023equivariant}\cite{franz2023index}, we also have (iv). 
\end{proof}

\begin{remark}\label{Rem: minimizer}
	By \cite{bray2010area}, every area minimizer $\wti\Sigma\in\wcS_{G_\pm}$ with $\Area(\wti\Sigma)=\sigma_0^{G_\pm}(S^3, g_{S^3})$ is the double cover of a stable (area minimizing) minimal projective plane $\Sigma\subset RP^3$. 
	Hence, $\wti\Sigma$ is a $G$-stable minimal $2$-sphere, which satisfies (iii) in the above theorem. 
\end{remark}

\subsection{$G_\pm$-Zoll metric on $S^3$}

\begin{definition}
	Let $\{\wti\Sigma_a\}_{a\in\mb {RP}^3}$ be a family of smoothly embedded $2$-spheres in $S^3$ so that $a\in\mb{RP}^3\mapsto \wti\Sigma_a\subset S^3$ is a $C^1$ map into the space of $C^{3,\alpha} $ embeddings of $2$-spheres, where $\alpha\in (0,1)$. 
	Then, $\{\wti\Sigma_a\}_{a\in\mb {RP}^3}$ is said to be a {\em Zoll family} in $S^3$ if for each $p\in S^3$ and each $2$-dimensional subspace $P\subset T_pS^3$, there exists a unique $a\in \mb {RP}^3$ with $p\in\wti\Sigma_a$ and $T_p\wti\Sigma_a=P$. 
	
	Additionally, if we further have $\{\wti\Sigma_a\}_{a\in\mb {RP}^3}\subset \wcS_{G_\pm}$, then we say $\{\wti\Sigma_a\}_{a\in\mb {RP}^3}$ is a {\em $G_\pm$-Zoll family}. 
	A $G$-invariant Riemannian metric $g_{S^3}$ on $S^3$ is called a {\em surface $G_\pm$-Zoll metric} if there is a $G_\pm$-Zoll family $\{\wti\Sigma_a\}_{a\in\mb {RP}^3}$ of minimal $2$-spheres for $(S^3,g_{S^3})$. 
\end{definition}

Note that without the $G=\mb Z_2$ action, the Zoll family and surface Zoll metric on $S^3$ are defined in \cite{ambrozio2024rigidity}, which can be easily generalized to $RP^3$ (cf. Theorem \ref{main Thm: zoll metric on RP3}). 
In particular, a $G_\pm$-Zoll family $\{\wti\Sigma_a\}_{a\in\mb {RP}^3}$ in $S^3$ naturally induces a Zoll family $\{\pi(\wti\Sigma_a)\}_{a\in\mb {RP}^3}\subset \mc S$ in $RP^3$. 
Therefore, a surface $G_\pm$-Zoll metric on $S^3$ induces a surface Zoll metric on $RP^3$.

The following result follows directly from \cite{ambrozio2024rigidity}. 
\begin{proposition}[\cite{ambrozio2024rigidity}*{Proposition 5.1}]\label{Prop: Zoll family}
	Let $\{\wti\Sigma_a\}_{a\in\mb {RP}^3} $ be a $G_\pm$-Zoll family of minimal $2$-spheres in $(S^3,g_{S^3})$. 
	Then 
	\begin{itemize}
		\item[(i)] $\Index(\wti\Sigma_a)=1$ and $\Null(\wti\Sigma_a)=3$ for each $a\in\mb {RP}^3$;
		\item[(ii)] $\wti\Sigma_a \cap \wti\Sigma_b\neq \emptyset$ is a smooth, connected, $G$-invariant simple closed curve, for $a\neq b\in \mb {RP}^3$. 
	\end{itemize}
\end{proposition}

Combining the above result with Proposition \ref{Prop: 2nd eigenfunctions}, we know $\tilde\lambda_1(\wti\Sigma_a)<0=\tilde\lambda_2(\wti\Sigma_a)=\tilde\lambda_1^G(\wti\Sigma_a)=\lambda_1(\pi(\wti\Sigma_a))$ for all $a\in\mb {RP}^3$.

\section{Rigidity for the $G_\pm$-spherical area widths}

In this section we prove the following rigidity result, which is equivalent to Theorem \ref{main Thm: zoll metric on RP3}. 
\begin{theorem}\label{Thm: rigidity}
	Let $g_{S^3}$ be a $G$-invariant Riemannian metric on $S^3$, where $G=\mb Z_2$. Then, 
	\begin{align}\label{Eq: constant spherical widths}
		 \sigma_0^{G_\pm}(S^3,  g_{S^3})=\sigma_1^{G_\pm}(S^3,  g_{S^3})=\sigma_2^{G_\pm}(S^3,  g_{S^3})=\sigma_3^{G_\pm}(S^3,  g_{S^3})
	\end{align}
	if and only if $g_{S^3}$ is a surface $G_\pm$-Zoll metric on $S^3$. 
\end{theorem}

\begin{proof}
	Using Proposition \ref{Prop: 2nd eigenfunctions} and Proposition \ref{Prop: nullily and graph}, we can adapt the proof in \cite{ambrozio2024rigidity}*{\S 5, \S 6} to fit our symmetric constraints. 
	As many of the ingredients are rather delicate, we provide a relatively self-contained proof for the sake of completeness. 
	
	\medskip
	\noindent{\bf Part I.} We first show the `if' part in Theorem \ref{Thm: rigidity}. 
	
	Suppose $\{\wti\Sigma_a\}_{a\in\mb {RP}^3} \subset \wcS_{G_\pm}$ is a $G_\pm$-Zoll family of minimal $2$-spheres in $(S^3,g_{S^3})$. 
	Then, since each $\wti\Sigma_a$ is minimal, the function $A(a):=\Area(\wti\Sigma_a)$ on $\mb {RP}^3$ has zero derivative, and thus $\Area(\wti\Sigma_a)\equiv W$ is a constant. 
	Let $\{\wti\Sigma_i^{(k)}\}_{i=0}^{l_k}$ be the disjoint family of $G$-invariant minimal $2$-spheres in Theorem \ref{Thm: min-max} for $0\leq k\leq 3$. 
	It follows from the uniqueness result of Galvez-Mira \cite{galvez2020uniquesness} that every component of the minimal surface $\wti\Sigma_i^{(k)}$ is an element of the $G_\pm$-Zoll family. 
	Hence, by Proposition \ref{Prop: Zoll family} and Theorem \ref{Thm: min-max}(i)(ii), we have $l_k=0$, and $\wti\Sigma^{(k)}_0\in\wcS_{G_\pm}$ is unstable with multiplicity $m^{(k)}_0=1$, which implies that $\sigma_k^{G_\pm}(S^3,g_{S^3}) = \Area(\wti\Sigma^{(k)}_0)\equiv W$ for each $0\leq k\leq 3$. 
	
	\medskip
	\noindent{\bf Part II.} We next show the `only if' part. 
	
	Suppose $\sigma_k^{G_\pm}(S^3,g_{S^3})\equiv W$ for all $0\leq k\leq 3$. Define 
    \[\mc V^G:=\{|\wti\Sigma|\in \mc V_2(S^3): \wti\Sigma\in\wcS_{G_\pm} \mbox{ is minimal in $(S^3,g_{S^3})$ with $\Area(\wti\Sigma)=W$} \}.\]
    Note that for any $|\wti\Sigma|\in\mc V^G$, $\wti\Sigma$ is area minimizing in $\wcS_{G_\pm}$, and thus $\tilde\lambda_2(\wti\Sigma)=\tilde\lambda_1^G(\wti\Sigma)\geq 0$, which implies that $\Index(\wti\Sigma)\leq 1$ by Proposition \ref{Prop: 2nd eigenfunctions}. 

    {\bf Step 1.} 
    \textit{
    Construct a min-max sequence $\{\Phi_i\}_{i\in\mb N}\subset\mc P_3^{G_\pm}$ with $X_i:={\rm dmn}(\Phi_i)$ so that
    \begin{itemize}
        \item[(i)] $\wti{\mc V}^G:=\mf C(\{\Phi_i\}_{i\in\mb N})=\mf \Lambda(\{\Phi_i\}_{i\in\mb N})\subset \mc V^G$;
        \item[(ii)] $|\Phi_i|(X_i)$ converges to $\wti{\mc V}^G$ uniformly in ($\mc V_2(S^3), \mf F$), i.e. $\sup_{x\in X_i}\mf F(|\Phi_i|(x), \wti{\mc V}^G) \to 0$;
        \item[(iii)] every $V\in \wti{\mc V}^G$ is induced by a smooth embedded $g_{S^3}$-minimal $\wti\Sigma\in\wcS_{G_\pm}$ with $\Index (\wti\Sigma)=1$, $1\leq\Null(\wti\Sigma)\leq 3$, and $0=\tilde\lambda_2(\wti\Sigma)=\tilde\lambda_1^G(\wti\Sigma)=\lambda_1(\pi(\wti\Sigma))$.
    \end{itemize}
    }
    
	Let $\{\Phi_i\}\subset\mc P_3^{G_\pm}$ be any min-max sequence for $\sigma_3^{G_\pm}(S^3,g_{S^3})=W$ with $X_i:={\rm dmn}(\Phi_i)$. By restricting to the $3$-skeleton of a certain component of $X_i$, we can assume without loss of generality that every $X_i$ is connected and $\dim(X_i)=3$. (See, for instance, \cite{marques2016morse}*{\S 1.5}.)
    
    Then, by the proof of \cite{ambrozio2024rigidity}*{Lemma 7.2}, we have 
    \[\lim_{i\to\infty} \sup_{x\in X_i} \mf F(|\Phi_i(x)|,\mc V^G) = 0,\]
    which shows $\wti{\mc V}^G:=\mf C(\{\Phi_i\}_{i\in\mb N})=\mf \Lambda(\{\Phi_i\}_{i\in\mb N})\subset \mc V^G$ in (i). 
    Indeed, suppose there exists $\epsilon>0$ and $x_{i_j}\in X_{i_j}$ so that $\mf F(|\Phi_{i_j}(x_{i_j})| , \mc V^G) > \epsilon$.  
    Then, note that
    \begin{align*}
        W=\sigma_0^{G_\pm}(S^3,g_{S^3})\leq \limsup_{j\to\infty}\Area(\Phi_{i_j}(x_{i_j})) \leq \limsup_{j\to\infty}\sup_{x\in X_{i_j}}\Area(\Phi_{i_j}(x)) =\sigma_3^{G_\pm}(S^3,g_{S^3})=W.
    \end{align*} 
    Thus, it follows from \cite{bray2010area} that $\pi(\Phi_{i_j}(x_{i_j}))$ converges (up to a subsequence) to an area minimizing minimal projective plane $\Sigma \subset RP^3$ with area $W/2$ in the varifold sense, which contradicts $\mf F(|\Phi_{i_j}(x_{i_j})|, |\pi^{-1}(\Sigma)|)>\epsilon$. 
    
    Next, if (ii) fails, then we have an $\epsilon>0$ and a sequence $x_{i_j}\in X_{i_j}$ so that $\mf F(|\Phi_{i_j}(x_{i_j})|, \wti{\mc V}^G) >\epsilon$. 
    Up to a subsequence, $|\Phi_{i_j}(x_{i_j})|$ converges to some $V\in \mf \Lambda(\{\Phi_i\}_{i\in\mb N})$, which contradicts $\mf \Lambda(\{\Phi_i\}_{i\in\mb N}) = \wti{\mc V}^G$ in (i). 
    Finally, to show (iii), we need the following result.

    \begin{claim}\label{Claim: step1}
        For any $V\in\wti{\mc V}^G$, there are pairwise distinct $V_1,V_2,\dots\in \wti{\mc V}^G$ so that $\lim_{i\to\infty}V_i=V$.
    \end{claim}
    \begin{proof}[Proof of Claim \ref{Claim: step1}]
        Otherwise, suppose there exists $\epsilon>0$ so that $B^{\bf F}_\epsilon(V)\cap \wti{\mc V}^G = \{V\}$. 
        Since $V\in \wti{\mc V}^G$, we have a sequence $x_{i_j}\in X_{i_j}$ with $V=\lim_{j\to\infty}|\Phi_{i_j}(x_{i_j})|$. 
        By (ii), we can take $j$ large enough so that $|\Phi_{i_j}|(X_{i_j})\subset B^{\mf F}_{\epsilon/5}(\wti{\mc V}^G)$ and $\mf F(|\Phi_{i_j}(x_{i_j})|, V)\leq\epsilon/5$.

        Consider the non-empty set $Y:=\{x\in X_{i_j}: \mf F(|\Phi_{i_j}(x)|, V)\leq\epsilon/5\}$ containing $x_{i_j}$, which is clearly closed. 
        Then for any $y_0\in Y$, we have $\mf F(|\Phi_{i_j}(y)|, V)\leq 2\epsilon/5$ for all $y\in X_{i_j}$ sufficiently close to $y_0$.  
        Additionally, by the choice of $j$, there exists $V_y\in \wti{\mc V}^G$ with $\mf F(|\Phi_{i_j}(y)|, V_y)< \epsilon/5$. Hence, we must have $V_y\in B^{\mf F}_\epsilon(V)\cap \wti{\mc V}^G=\{V\}$, which implies $Y$ is also open. 
        This shows $Y=X_{i_j}$, and thus $\mf C(\{\Phi_{i_j}\}_{j\in\mb N})=\mf \Lambda(\{\Phi_{i_j}\}_{j\in\mb N})=\{V\}$. 

        However, given $U\subset\subset S^3\setminus\spt(\|V\|)$, it follows from \cite{marques2017existence}*{Proposition 8.2} that $\Area(\Phi_{i_j}(y_{i_j})\cap U)>C$ for some $y_{i_j}\in X_{i_j}$ and constant $C=C(U,g_{S^3})>0$, which is a contradiction. 
    \end{proof}
    
    Let $V=|\wti\Sigma|$ and $V_i=|\wti\Sigma_i|$ be given as in Claim \ref{Claim: step1}. 
    Then by Allard regularity \cite{allard1972first}, $\wti\Sigma_i$ converges smoothly to $\wti\Sigma$. 
    Hence, for $i$ large enough, $\wti\Sigma_i= \Graph(u_i):=\exp_{\wti\Sigma}^\perp(u_i\tilde\nu)$ for some $u_i\in C^\infty(\wti\Sigma)$. 
    Additionally, since $\wti\Sigma,\wti\Sigma_i\in\wcS_{G_\pm}$, we know $u_i\circ g_-=-u_i$. 
    Therefore, we can apply \cite{sharp2017compactness}*{Theorem 2.3} (see also \cite{wang2023equivariant}*{Theorem 5.6}) to obtain a nontrivial solution $v$ of the Jacobi equation $L_{\wti\Sigma} v=0$ so that $v\circ g_-=-v$, which implies that $0$ is an eigenvalue of $\pi(\wti\Sigma)$, i.e.  $0=\tilde\lambda_{k_0}^G(\wti\Sigma)=\lambda_{k_0}(\pi(\wti\Sigma))$ for some $k_0\geq 1$. 
    Note also that $\wti\Sigma$ with area $W=\sigma_0^{G_\pm}(S^3,g_{S^3})$ is area minimizing in $\wcS_{G_\pm}$, so $\tilde\lambda_1^G(\wti\Sigma)=\lambda_1(\pi(\wti\Sigma))\geq 0$. 
    Together, we see $0=\lambda_1(\pi(\wti\Sigma))$, and (iii) follows from Proposition \ref{Prop: 2nd eigenfunctions}.

    \medskip
    \noindent{\bf Step 2.}
    \textit{Given $|\wti\Sigma|\in \wti{\mc V}^G$, we can perturb each $\Phi=\Phi_i$ with $i$ large enough to have a new spherical $(G_\pm, 3)$-sweepout $\Phi'$ so that $\Phi'(x)$ is a graph over $\wti\Sigma$ given by Proposition \ref{Prop: nullily and graph} whenever $|\Phi'(x)|$ is close to $|\wti\Sigma|$.}

    Firstly, set 
    \[ \wti{\mc T}^G :=\{ \wti\Sigma\in \wcS_{G_\pm} : |\wti\Sigma|\in\wti{\mc V}^G \}.  \]
    By Allard's regularity, $\wti{\mc T}^G$ can also be viewed as a compact set of $(\mc Z_2(S^3;\mb Z_2), \mf F)$. 
    Then, combining (ii) in {\bf Step 1} with the proof of \cite{ambrozio2024rigidity}*{Proposition 6.6}, we know 
    \begin{align}\label{Eq: uniform F-convergence to TG}
    	\lim_{i\to\infty}\sup_{x\in X_i}\mf F(\Phi_i(x), \wti{\mc T}^G) = 0.
    \end{align}
    For any fixed integer $k\geq 2$, $\alpha\in (0,1)$, $\delta>0$, and $|\wti\Sigma|\in \wti{\mc V}^G$, we use the following notation: 
    \begin{itemize}
		\item $\varphi_{\wti\Sigma}: \overline {\mb B}^3_1:=\{x\in\mb R^3: |x|\leq 1\}\to \wcS^{k,\alpha}_{G_\pm}$ and $\eta_{\wti\Sigma}>0$ are the smooth embedding and the constant given by Proposition \ref{Prop: nullily and graph} (note that $\varphi_{\wti\Sigma}$ can be extended from $\overline{\mb B}^{\Null({\wti\Sigma})}_1$ to $\overline{\mb B}^3_1$);
		\item $\mc W_{{\wti\Sigma}} := \varphi_{\wti\Sigma}(\overline{\mb B}^3_1)\subset  \wcS^{k,\alpha}_{G_\pm}$;
        \item $\mc U_\delta^{k,\alpha}({\wti\Sigma}) :=\{{\wti\Sigma}'\in \wcS^{k,\alpha}_{G_\pm}: {\wti\Sigma}'=\exp_{\wti\Sigma}^\perp(f\tilde\nu), |f|_{k,\alpha}<\delta\}$ is the $\delta$-neighborhood of ${\wti\Sigma}$ in $\wcS^{k,\alpha}_{G_\pm}$. 
    \end{itemize}
    We can further shrink $\eta_{\wti\Sigma}$ so that $\mf F(|{\wti\Sigma}|,|{\wti\Sigma}'|)\geq 2\eta_{\wti\Sigma}$ for any ${\wti\Sigma}'\in\bd\mc W_{\wti\Sigma} = \varphi_{\wti\Sigma}(\bd \overline {\mb  B}^3_1)$. 
    In particular, if ${\wti\Sigma}'$ is a minimal surface with $\mf F(|{\wti\Sigma}|,|{\wti\Sigma}'|)<\eta_{\wti\Sigma}$, then by Proposition \ref{Prop: nullily and graph} and the above assumption, we have ${\wti\Sigma}'=\varphi_{\wti\Sigma}(v)$ for some $v\in \mb B_{1-\epsilon}^3:=\{x\in\mb R^3: |x|<1-\epsilon\}$ and $\epsilon>0$. 

    \begin{claim}\label{Claim: constants}
        We have the following statements.
        \begin{itemize}
            \item[(i)] For any $\epsilon'>0$, there exists $\tilde\delta=\tilde\delta(\epsilon')>0$ so that if ${\wti\Sigma}_1\in \wti{\mc T}^G$ and $\Gamma\in \mc U^{k,\alpha}_{\tilde\delta}({\wti\Sigma}_1)$, then $\mf F({\wti\Sigma}_1,\Gamma) < \epsilon'$.
            \item[(ii)] For any $\delta>0$, there exists $\tilde\theta=\tilde\theta(\delta)\in (0,\delta)$ so that if ${\wti\Sigma}_1,{\wti\Sigma}_2\in \wti{\mc T}^G$ satisfy ${\wti\Sigma}_2\in \mc U^{k,\alpha}_{\tilde\theta}({\wti\Sigma}_1)$, then $\mc U^{k,\alpha}_{\tilde\theta}({\wti\Sigma}_1) \subset \mc U^{k,\alpha}_{\delta}({\wti\Sigma}_2) $.
            \item[(iii)] For any $\theta>0$, there exists $\tilde\xi=\tilde\xi(\theta)>0$ so that if ${\wti\Sigma}_1,{\wti\Sigma}_2\in \wti{\mc T}^G$ satisfy $\mf F({\wti\Sigma}_1,{\wti\Sigma}_2)<5\tilde\xi$, then ${\wti\Sigma}_2\in \mc U^{k,\alpha}_\theta({\wti\Sigma}_1)$.  
        \end{itemize}
    \end{claim}
    \begin{proof}[Proof of Claim \ref{Claim: constants}]
        (i) and (ii) follow from the compactness of $\wti{\mc T}^G$ and the fact that the $C^{k,\alpha}$-topology is stronger than the $\mf F$-topology on $\wcS_{G_\pm}^{k,\alpha}$. 
        Additionally, by Allard's regularity, we know the $\mf F$-topology and the $C^{k,\alpha}$-topology are equivalent on the closed set $\wti{\mc T}^G$. 
        Together with the compactness of $\wti{\mc T}^G$, we also have (iii). 
    \end{proof}

    Fix any $|{\wti\Sigma}|\in\wti{\mc V}^G$. Let $\bar \xi>0$ be determined later. 
    
    For each $i$ large enough, it follows from \eqref{Eq: uniform F-convergence to TG} that $\Phi:=\Phi_i: X:=X_i\to \wcS^{k,\alpha}_{G_\pm}$ satisfies 
    \begin{align}\label{Eq: step2 uniform close to TG}
        \sup_{x\in X} \mf F(\Phi(x), \wti{\mc T}^G) < \bar\xi.
    \end{align}
    Consider the barycentric subdivision $X(l)$ of $X$ so that 
    \begin{align}\label{Eq: step2 subdivision}
        \mf F(\Phi(x),\Phi(y)) := \mc F(\Phi(x),\Phi(y)) + \mf F(|\Phi(x)|, |\Phi(y)|) < \bar\xi,
    \end{align}
    for every face $\alpha\in X(l)$ and $x,y\in\alpha$. Then we define 
    \begin{align}\label{Eq: step2 define vertex}
        \Phi'(x) := {\wti\Sigma}_x \in \wti{\mc T}^G ~{\rm with}~\mf F(\Phi(x), {\wti\Sigma}_x)<\bar\xi, \quad \mbox{for any vertex $x$ of $X(l)$}.
    \end{align}
    Let $X':=\{x\in X: \mf F(|\Phi(x)|,|{\wti\Sigma}|) < \eta_{\wti\Sigma}/2\}$ and 
    \begin{align}\label{Eq: step 2 domain close to Sigma}
        \wti{X}:=\{\tau\in X(l): \exists~ \mbox{cell $\sigma\in X(l)$ with } \tau\subset\sigma,~ \sigma \cap X' \neq \emptyset\}. 
    \end{align}
    We can then extend $\Phi'$ from the vertex of $X(l)$ to the $1$-skeleton of $X(l)$. 
    
    \begin{claim}[Extends to $1$-skeleton]\label{Claim : extend to 1 cells}
        For any $\xi_1,\theta_1 >0$, we can shrink $\bar\xi>0$ small enough and $C^{k,\alpha}$-continuously extend $\Phi'$ to the $1$-skeleton of $X(l)$ so that for every $1$-cell $[x,y]\in X(l)$,
        \begin{itemize}
            \item[(i)] $\mf F(\Phi'(z),{\wti\Sigma}_{z_0}) < \xi_1$ and $\Phi'(z)\in \mc U^{k,\alpha}_{\theta_1}({\wti\Sigma}_{z_0})$, where $z_0\in\{x,y\}$ is any vertex and $z\in [x,y]$;
            \item[(ii)] if $[x,y]\in \wti X$, then $\Phi'(z)\in \mc W_{\wti\Sigma}$ for any $z\in [x,y]$. 
        \end{itemize}
    \end{claim}
    \begin{proof}[Proof of Claim \ref{Claim : extend to 1 cells}]
        Firstly, let 
        \begin{itemize}
            \item[(1)] $\delta=\tilde\delta(\xi_1) \in (0,\theta_1)$ be given by Claim \ref{Claim: constants}(i);
            \item[(2)] $\theta_0=\tilde\theta(\delta)\in (0,\delta)$ be given by Claim \ref{Claim: constants}(ii);
            \item[(3)] $0<\bar\xi<\tilde\xi(\theta)$, where $\tilde\xi(\theta)\in (0,\theta)$ is given by Claim \ref{Claim: constants}(iii), and $\theta\in (0,\theta_0)$ will be determined later.
        \end{itemize} 
        Notice that $\mf F({\wti\Sigma}_x,{\wti\Sigma}_y) < 3\bar\xi$ for any $1$-cell $[x,y]\in X(l)$ with vertices $x,y$. 
        Hence, by (3), we know ${\wti\Sigma}_y\in \mc U^{k,\alpha}_\theta({\wti\Sigma}_x)$. 
        In particular, ${\wti\Sigma}_y=\Graph_{{\wti\Sigma}_x}(h)$ for some $h\in C_{G_\pm}^{k,\alpha}({\wti\Sigma}_x)$ with $|h|_{k,\alpha}<\theta$. 

        If $[x,y]\notin \wti X$, then we define 
        \[\Phi'((1-t)x+ty ) := \Graph_{{\wti\Sigma}_x}(th)\in \mc U_\theta^{k,\alpha}({\wti\Sigma}_x)\subset \wcS_{G_\pm}^{k,\alpha} \qquad \mbox{for any $t\in [0,1]$}.\] 
        By (2) and $\theta<\theta_0<\delta<\theta_1$, we see that $\Phi'(z)\in \mc U_{\theta_0}^{k,\alpha}({\wti\Sigma}_x)\subset \mc U_{\delta}^{k,\alpha}({\wti\Sigma}_x)\cap \mc U_{\delta}^{k,\alpha}({\wti\Sigma}_y)\subset \mc U_{\theta_1}^{k,\alpha}({\wti\Sigma}_x)\cap \mc U_{\theta_1}^{k,\alpha}({\wti\Sigma}_y)$ for all $z\in [x,y]$. 
        Together with (1), we have $\mf F(\Phi'(z),{\wti\Sigma}_x)<\xi_1$ and $\mf F(\Phi'(z),{\wti\Sigma}_y)<\xi_1$ for all $z\in [x,y]$. 

        If $[x,y]\in\wti X$, then there is a cell $\sigma\in X(l)$ with $[x,y]\subset\sigma$ and some $w\in \sigma\cap X'$. 
        Thus, 
        \begin{align}\label{Eq: step2 vertex and given Sigma}
        		\mf F(|{\wti\Sigma}|,|{\wti\Sigma}_x|)\leq \mf F(|{\wti\Sigma}|,|\Phi(w)|)+ \mf F(|\Phi(w)|,|\Phi(x)|)+\mf F(|\Phi(x)|,|{\wti\Sigma}_x|) < \eta_{\wti\Sigma}/2 + 4\bar\xi.
        \end{align}
        By taking $4\bar\xi < \eta_{\wti\Sigma}/2$, we know ${\wti\Sigma}_x=\varphi_{\wti\Sigma}(v_x)\in\mc W_{\wti\Sigma}$ for some $v_x\in {\mb B}^3_{1-\epsilon}$ (by Proposition \ref{Prop: nullily and graph} and the assumption before Claim \ref{Claim: constants}). 
        Similarly, ${\wti\Sigma}_y=\varphi_{\wti\Sigma}(v_y)\in\mc W_{\wti\Sigma}$ for some $v_y\in {\mb B}^3_{1-\epsilon}$. 
        Note that $\varphi_{\wti\Sigma}: \overline {\mb B}^3\to \wcS^{k,\alpha}_{G_\pm}$ is an embedding. 
        Hence, we can take $\theta>0$ small enough so that $(1-t)v_x+tv_y\in\mb B_{1-\epsilon}^3$ and $\varphi_{\wti\Sigma}((1-t)v_x+ tv_y)\in \mc U_{\theta_0}^{k,\alpha}({\wti\Sigma}_x)$ for all $t\in[0,1]$. 
        Define then 
        \[\Phi'((1-t)x+ty):=\varphi_{\wti\Sigma}((1-t)v_x+ tv_y)\in \mc U_{\theta_0}^{k,\alpha}({\wti\Sigma}_x) \qquad \mbox{for any $t\in[0,1]$}.\] 
        By (1)(2) and $\theta_0<\delta<\theta_1$, we still have $\Phi'(z)\in \mc U_{\theta_1}^{k,\alpha}({\wti\Sigma}_x)\cap \mc U_{\theta_1}^{k,\alpha}({\wti\Sigma}_y)$, $\mf F(\Phi'(z),{\wti\Sigma}_x)<\xi_1$, and $\mf F(\Phi'(z),{\wti\Sigma}_y)<\xi_1$, for all $z\in [x,y]$. 
    \end{proof}

    By a similar procedure, we can further extend $\Phi'$ to the $2$-skeleton of $X(l)$. 
    \begin{claim}[Extends to $2$-skeleton]\label{Claim : extend to 2 cells}
        For any $\xi_2,\theta_2 >0$, we can shrink $\bar\xi>0$ small enough and $C^{k,\alpha}$-continuously extend $\Phi'$ to the $2$-skeleton of $X(l)$ so that for any $2$-cell $\tau\in X(l)$,
        \begin{itemize}
            \item[(i)] $\mf F(\Phi'(z),{\wti\Sigma}_x) < \xi_2$ and $\Phi'(z)\in \mc U_{\theta_2}^{k,\alpha}({\wti\Sigma}_x)$, where $x$ is any vertex of $\tau$ and $z\in\tau$;
            \item[(ii)] if $\tau\in \wti X$, then $\Phi'(z)\in \mc W_{\wti\Sigma}$ for any $z\in\tau$. 
        \end{itemize}
    \end{claim}
    \begin{proof}[Proof of Claim \ref{Claim : extend to 2 cells}]
        Firstly, by Claim \ref{Claim: constants} and the embeddedness of $\varphi_{\wti\Sigma}$, let 
        \begin{itemize}
            \item[(1)] $\delta=\tilde\delta(\xi_2)\in (0,\theta_2)$ be given by Claim \ref{Claim: constants}(i);
            \item[(2)] $\theta_0=\tilde\theta(\delta)\in (0,\delta)$ 
            be given by Claim \ref{Claim: constants}(ii);
            \item[(3)] $\theta_1'\in (0, \theta_0)$ so that if $v_1,v_2\in {\mb B}^3_{1-\epsilon}$ satisfy $\varphi_{\wti\Sigma}(v_2)\in \mc U^{k,\alpha}_{\theta_1'}(\varphi_{\wti\Sigma}(v_1))$, then $(1-t)v_1+tv_2\in\mb B^3_{1-\epsilon}$ and $\varphi_{\wti\Sigma}((1-t)v_1+tv_2)\in \mc U^{k,\alpha}_{\theta_0}(\varphi_{\wti\Sigma}(v_1))$ for all $t\in [0,1]$, where $\epsilon>0$ is fixed before Claim \ref{Claim: constants};
            \item[(4)] $\theta_1=\tilde\theta(\theta_1')\in (0,\theta_1')$ be given by Claim \ref{Claim: constants}(ii);
            \item[(5)] $\xi_1\in (0,\xi_2)$ so that $4\xi_1<\eta_{\wti\Sigma}/2$. 
        \end{itemize}
        Next, let $\Phi'$ be extended as in Claim \ref{Claim : extend to 1 cells} with respect to the constants $\xi_1$ in (5) and $\theta_1$ in (4). 
        Then for any $2$-cell $\tau$ of $X(l)$, take vertices $x_1,x_2,x_3$, so that $[x_1,x_2],[x_1,x_3]\subset\tau$. 
        By Claim \ref{Claim : extend to 1 cells}(i) and (4), we conclude that $\wti\Sigma_{x_1}\in \mc U^{k,\alpha}_{\theta_1}({\wti\Sigma}_{x_2})\cap \mc U^{k,\alpha}_{\theta_1}({\wti\Sigma}_{x_3})$, and thus 
        \[\Phi'(z)\in \mc U^{k,\alpha}_{\theta_1}({\wti\Sigma}_{x_2})\cup \mc U^{k,\alpha}_{\theta_1}({\wti\Sigma}_{x_3}) \subset \mc U^{k,\alpha}_{\theta_1'}({\wti\Sigma}_{x_1}),\qquad \forall z\in\bd\tau.\]
        In particular, for any vertex $x$ of $\tau$, we have $\wti\Sigma_x\in \mc U^{k,\alpha}_{\theta_1'}({\wti\Sigma}_{x_1})\subset\mc U^{k,\alpha}_{\theta_0}({\wti\Sigma}_{x_1}) \subset \mc U^{k,\alpha}_{\delta}({\wti\Sigma}_{x})$. 

        If $\tau\notin\wti X$, then for all $z\in\bd\tau$, there exists $h_z\in C^{k,\alpha}_{G_\pm}({\wti\Sigma}_{x_1})$ with $|h_z|_{k,\alpha}<\theta_0$ so that $\Phi'(z)=\Graph_{\wti\Sigma_{x_1}}(h_z)$. 
        Let $\phi:\overline{\mb B}^2\to\tau$ be a homeomorphism, and define 
        \[\Phi'(\phi(tu)) := \Graph_{\wti\Sigma_{x_1}}(t h_{\phi(u)})\in \mc U_{\theta_0}^{k,\alpha}({\wti\Sigma}_{x_1}) \qquad \mbox{ for any $t\in [0,1]$ and $u\in\bd\overline{\mb B}^2$}.\] 
        Hence, by (2) and $\theta_0<\delta<\theta_2$, we conclude that $\Phi'(z)\in \mc U_{\theta_0}^{k,\alpha}({\wti\Sigma}_{x_1})\subset \mc U_\delta^{k,\alpha}({\wti\Sigma}_{x})\subset \mc U_{\theta_2}^{k,\alpha}({\wti\Sigma}_{x})$ for any $z\in \tau$ and vertex $x$ of $\tau$, which further implies $\mf F(\Phi'(z),{\wti\Sigma}_x)<\xi_2$ by (1).

        If $\tau\in\wti X$, then by Claim \ref{Claim : extend to 1 cells}(ii) and its proof, for any $z\in\bd \tau$, there exists $v_z\in \mb B^3_{1-\epsilon}$ so that $\Phi'(z)=\varphi_{\wti\Sigma}(v_z)\in \mc W_{\wti\Sigma}$. 
        In particular, ${\wti\Sigma}_{x_1}=\varphi_{\wti\Sigma}(v_1)\in\mc W_{\wti\Sigma}$ for some $v_1\in {\mb B}^3_{1-\epsilon}$. 
        Since $\Phi'(\bd\tau)\subset \mc U^{k,\alpha}_{\theta_1'}({\wti\Sigma}_{x_1})$, we can use (3) and the homeomorphism $\phi:\overline{\mb B}^2\to\tau$ to define 
        \[\Phi'(\phi(tu)):=\varphi_{\wti\Sigma}((1-t)v_1+ tv_{\phi(u)})\in \mc U^{k,\alpha}_{\theta_0}({\wti\Sigma}_{x_1})\qquad \mbox{ for all $t\in [0,1]$ and $u\in\bd \overline{\mb B}^2$},\] 
        which gives (ii). 
        Again, (i) follows from (1)(2) and the fact that $\theta_0<\delta<\theta_2$. 
    \end{proof}

    The above construction is slightly different from the proof of \cite{ambrozio2024rigidity}*{Proposition 6.7}, and is also valid when $X$ is a cubical complex. 
	By a very similar construction, one can further extend $\Phi'$ to the whole $3$-dimensional complex $X$. 
	We leave the proof to readers. 
    \begin{claim}[Extends to $X$]\label{Claim : extend to 3 cells}
        For any $\xi_3,\theta_3 >0$, we can shrink $\bar\xi>0$ small enough and $C^{k,\alpha}$-continuously extend $\Phi'$ to $X$ so that for every $3$-cell $\tau\in X(l)$ and $z\in \tau$,
        \begin{itemize}
            \item[(i)] $\mf F(\Phi'(z),{\wti\Sigma}_x) < \xi_3$ and $\Phi'(z)\in \mc U_{\theta_3}^{k,\alpha}({\wti\Sigma}_x)$, where $x$ is any vertex of $\tau$;
            \item[(ii)] if $\tau\in \wti X$, then $\Phi'(z)\in \mc W_{\wti\Sigma}$. 
        \end{itemize}
    \end{claim}

    Finally, we show the following main result in this step. 
    \begin{claim}[Local graphical perturbation]\label{Claim: local graphical perturbation}
        Given $|{\wti\Sigma}|\in\wti{\mc V}^G$ and $\xi>0$, there exists $0<\bar\xi<\xi/2$ so that if $\Phi_i: X_i\to \wcS^{k,\alpha}_{G_\pm}$ satisfies \eqref{Eq: step2 uniform close to TG}, then we can find a continuous map $\Phi_i':X_i\to\wcS^{k,\alpha}_{G_\pm}$ satisfying $\Phi_i'\in\mc P_3^{G_\pm}$, and
        \begin{itemize}
            \item[(i)] $\mf F(\Phi_i'(x),\Phi_i(x))<2\xi$;
            \item[(ii)] for any $x\in X_i$, either $\mf F(|\Phi_i'(x)|, |{\wti\Sigma}|)>\eta_{\wti\Sigma}/3$ or $\Phi_i'(x)\in \mc W_{\wti\Sigma}$. 
        \end{itemize}
    \end{claim}
    \begin{proof}[Proof of Claim \ref{Claim: local graphical perturbation}]
        Let $0<\xi_3=\theta_3<\min\{\xi, \delta/2,\eta_{\wti\Sigma}/12\}$, where $\delta>0$ is given by \cite{ambrozio2024rigidity}*{Lemma 2.7}. 
        Consider the barycentric subdivision $X_i(l_i)$ so that \eqref{Eq: step2 subdivision} is satisfied and define $X_i',\wti X_i$ as in \eqref{Eq: step 2 domain close to Sigma}. 
        Then by Claim \ref{Claim : extend to 3 cells}, we take $\bar\xi< \xi_3/2 $ and obtain a continuous map $\Phi_i': X_i\to \wcS^{k,\alpha}_{G_\pm}$. 
        For any $3$-cell $\tau\in X_i(l_i)$ with $z\in\tau$ and vertex $x\in\tau$, we have $\mf F(\Phi_i'(z),{\wti\Sigma}_x)<\xi_3$, and by \eqref{Eq: step2 subdivision}\eqref{Eq: step2 define vertex},
        \[ \mf F(\Phi_i'(z),\Phi_i(z)) \leq \mf F(\Phi_i'(z), {\wti\Sigma}_x) + \mf F({\wti\Sigma}_x,\Phi_i(x)) + \mf F(\Phi_i(x),\Phi_i(z)) < 2\xi_3 < \min\{2\xi,\delta,\eta_{\wti\Sigma}/6\}.  \]
        In particular, this shows (i) and $\Phi_i'\in\mc P^{G_\pm}_3$ by \cite{ambrozio2024rigidity}*{Lemma 2.7}. 
        Additionally, by Claim \ref{Claim : extend to 3 cells}(ii), we know $\Phi_i'(\wti X_i)\subset \mc W_{\wti\Sigma}$. 
        Meanwhile, for any $z\in X_i\setminus \wti X_i$, we have $z\notin X_i'$ and $\mf F(|\Phi_i(z)|, |{\wti\Sigma}|)\geq \eta_{\wti\Sigma}/2$. Hence, $\mf F(|\Phi_i'(z)|, |{\wti\Sigma}|)\geq \mf F( |{\wti\Sigma}|,|\Phi_i(z)|)-\mf F(\Phi_i(z),\Phi_i'(z))> \eta_{\wti\Sigma}/3$.
        This shows (ii). 
    \end{proof}


    \medskip
    \noindent{\bf Step 3.} \textit{Let
    \begin{itemize}
        \item $\mc G\subset \wti{\mc V}^G$ be the set of $|{\wti\Sigma}|\in\wti{\mc V}^G$ satisfying that $\varphi_{\wti\Sigma}$ and $\eta_{\wti\Sigma}$ (in Proposition \ref{Prop: nullily and graph}) can be chosen so that $\varphi_{\wti\Sigma}(v)\in  \wcS^{k,\alpha}_{G_\pm}$ is a minimal surface (and thus smooth) for all $v\in\overline{\mb B}^3_1$;
        \item $\mc B := \wti{\mc V}^G\setminus \mc G$, which is compact in the $C^{k,\alpha}$-topology (since $\mc G$ is open). 
    \end{itemize}
    We can perturb $\mf \Lambda (\{\Phi_i\})$ away from $\mc B$, and obtain a new min-max sequence $\wti\Psi_i: \wti{Y}_i^3\to \wcS^{k,\alpha}_{G_\pm}$ so that $\{\wti \Psi_i\}_{i\in\mb N}\subset \mc P^{G_\pm}_3$ and $\wti\Psi_i(\wti{Y}_i)$ converges uniformly to a compact set 
    \[\wti {\mc G}:= \mf \Lambda(\{\wti \Psi_i\}_{i\in\mb N}) \subset \wti{\mc V}^G\setminus\mc B=\mc G\]
    in the $\mf F$-topology, i.e. $\lim_{i\to\infty} \sup_{x\in\wti Y_i}\mf F(|\wti\Psi_i(x)|,\wti{\mc G}) = 0$.}

    By the definitions, for any ${\wti\Sigma}\in\mc B$, there exists $v_i\to 0\in \mb B^3$ so that $\varphi_{\wti\Sigma}(v_i)$ is not minimal for each $i$. 
    Hence, given ${\wti\Sigma}\in\mc B$, we can find $v_{\wti\Sigma} \in \mb B^3_1$ and $\gamma_{\wti\Sigma}>0$ so that 
    \begin{itemize}
        \item[(a)] $\varphi_{\wti\Sigma}(v_{\wti\Sigma})$ is not minimal;
        \item[(b)] $t_{\wti\Sigma}:= |v_{\wti\Sigma}|$ satisfies $\mf F(|\varphi_{\wti\Sigma}(v)|, |{\wti\Sigma}|)<\eta_{\wti\Sigma}/8$ for any $v\in \mb B^3_{t_{\wti\Sigma}}$;
        \item[(c)] $\mf F(|{\wti\Sigma}|, |\varphi_{\wti\Sigma}(v)|) \geq \gamma_{\wti\Sigma}$ for any $v\in \overline{\mb B}^3_1 \setminus \mb B^3_{t_{\wti\Sigma}/2}$. 
    \end{itemize}
    
    Using the perturbation procedure in Step 2, we first show the following result to perturb $\mf \Lambda (\{\Phi_i\})$ away from a given ${\wti\Sigma}\in\mc B$. 
    \begin{claim}[Perturb away from ${\wti\Sigma}\in\mc B$]\label{Claim: step3 perturb away from S in B}
        Given ${\wti\Sigma}\in\mc B$, there exists a sequence of continuous maps $\Psi_i: Y_i^3\to\wcS^{k,\alpha}_{G_\pm}$ so that $\{\Psi_i\}_{i=1}^\infty\subset \mc P^{G_\pm}_{3}$ and 
        \[\mf \Lambda(\{\Psi_i\}_{i\in\mb N}) \subset \mf \Lambda(\{\Phi_i\}_{i\in\mb N}) \setminus  B^{\mf F}_{\gamma_{\wti\Sigma}}(|{\wti\Sigma}|).  \]
    \end{claim}
    \begin{proof}[Proof of Claim \ref{Claim: step3 perturb away from S in B}]
        Since $\varphi_{\wti\Sigma}(v_{\wti\Sigma})$ is not minimal, we can take $0<3\xi <\mf F(\varphi_{\wti\Sigma}(v_{\wti\Sigma}), \wti {\mc T}^G)$. 
        Then for each $i$ sufficiently large, we can apply Claim \ref{Claim: local graphical perturbation} to $\Phi_i$ satisfying \eqref{Eq: step2 uniform close to TG}, and obtain a spherical $(G_\pm,3)$-sweepout $\Phi'_i: X_i\to \wcS^{k,\alpha}_{G_\pm}$. 
        By Claim \ref{Claim: local graphical perturbation}(i) and \eqref{Eq: step2 uniform close to TG}, we know $\Phi'_i(X_i)\subset B^{\mf F}_{3\xi}(\wti{\mc T}^G)$, and thus $\varphi_{\wti\Sigma}(v_{\wti\Sigma})\notin \Phi'_i(X_i)$ since $3\xi< \mf F(\varphi_{\wti\Sigma}(v_{\wti\Sigma}), \wti {\mc T}^G)$. 
        Therefore, $\varphi_{\wti\Sigma}^{-1}(\Phi'_i(X_i)) \cap \bd \mb B^3_{t_{\wti\Sigma}}$ has an open neighborhood $\Omega$ in the $2$-sphere $\bd \mb B^3_{t_{\wti\Sigma}}$ with $v_{\wti\Sigma}\notin\closure(\Omega)$. 
        Note that $H^k(\Omega;\mb Z_2)=0$ for all $k\geq 2$ since $\Omega\subset\subset \bd \mb B^3_{t_{\wti\Sigma}}\setminus\{v_{\wti\Sigma}\}$. 

        Next, let $\mc X_2:=\varphi_{\wti\Sigma}(\mb B^3_{t_{\wti\Sigma}})$. As in the proof of \cite{ambrozio2024rigidity}*{Proposition 6.8}, we can choose an open neighborhood $\mc X_1$ of $\Phi_i'(X_i)\setminus\varphi_{\wti\Sigma}(\mb B^3_{t_{\wti\Sigma}})$ in $\{S\in\wcS^{k,\alpha}_{G_\pm}: \mf F(|S|,|{\wti\Sigma}|)>\eta_{\wti\Sigma}/3\}\cup \mc W_{\wti\Sigma}$ under the $\mf F$-metric of mod $2$ cycles so that 
        \[ \varphi_{\wti\Sigma}(\{sv: v\in\Omega, s\in(1-\delta,1)\}) = \mc X_1\cap\mc X_2 \qquad \mbox{for some $\delta>0$}. \]
        Thus, $H^3(\mc X_2;\mb Z_2)=0$, and $H^k(\mc X_1\cap\mc X_2;\mb Z_2)=0$ for $k\geq 2$. 
        Additionally, by Claim \ref{Claim: local graphical perturbation}(ii), we know $\Phi_i'(X_i)\subset \mc X:=\mc X_1\cup\mc X_2$. 
        It then follows from the Mayer-Vietoris sequence that $i_1^*: H^3(\mc X;\mb Z_2)\to H^3(\mc X_1;\mb Z_2)$ is an isomorphism, where $i_1:\mc X_1\to\mc X$ is the inclusion map. 
        Hence, there exists a $3$-dimensional complex $Y_i$ and a spherical $(G_\pm,3) $-sweepout $\Psi_i: Y_i\to \mc X_1\subset \wcS^{k,\alpha}_{G_\pm}$. 
        By taking a sequence of $\xi\to 0$, the desired result follows from a diagonal argument. 
    \end{proof}

    Note that $\mc B\subset \wti{\mc V}^G$ is compact in the $\mf F$-topology. 
    Thus, there exists $\{{\wti\Sigma}_1,\dots,{\wti\Sigma}_m\}\subset\mc B$ with $\mc B\subset\cup_{i=1}^m B^{\mf F}_{\gamma_{{\wti\Sigma}_i}}(|{\wti\Sigma}_i|) $, where $\gamma_{\wti\Sigma_i}$ is given by (c). 
    We can then apply Claim \ref{Claim: step3 perturb away from S in B} successively for each $\{{\wti\Sigma}_i\}_{i=1}^m$, and obtain a min-max sequence $\wti\Psi_i:\wti Y_i\to\wcS^{k,\alpha}_{G_\pm}$ so that $\{\wti \Psi_i\}_{i\in\mb N}\subset \mc P^{G_\pm}_3$ and $\wti{\mc G}:=\mf \Lambda(\{\wti\Psi_i\}_{i\in\mb N}) \subset \mf \Lambda(\{\Phi_i\}_{i\in\mb N})\setminus \mc B$.

    \medskip
    {\noindent \bf Step 4.} {\it $\wti\Psi_i$ can be further perturbed to a spherical $(G_\pm,3)$-sweepout $\Psi_i:\wti Y_i\to \mc G$ so that $\lim_{i\to\infty} \sup_{x\in\wti Y_i}\mf F(|\Psi_i(x)|,\wti{\mc G}) = 0$.}

    We first show that $\mc G$ is a $3$-manifold, then the geodesic ball on $\mc G$ will help us to construct $\Psi_i$ as in Step 2. 
    \begin{claim}\label{Claim: step4 G is manifold}
        $\mc G$ is a smooth $3$-manifold with local chart given by $\varphi_{\wti\Sigma}$. Additionally, every $\wti\Sigma\in\mc G$ has index $1$ and nullity $3$. 
    \end{claim}
    \begin{proof}[Proof of Claim \ref{Claim: step4 G is manifold}]
        Given ${\wti\Sigma}\in\mc G$, since $\varphi_{\wti\Sigma}(v)$ is minimal for all $v\in\overline{\mb B}^3_1$, the function $v\mapsto \Area(\varphi_{\wti\Sigma}(v))$ has vanished derivative, which implies $\Area(\varphi_{\wti\Sigma}(v))\equiv \Area({\wti\Sigma})=W$. 
        Additionally, for any $v\in \overline{\mb B}^3_1$, the embeddedness of $\varphi_{\wti\Sigma}$ implies that $\Null(\varphi_{\wti\Sigma}(v))\geq 3$, and $0$ is an equivariant eigenvalue $\tilde\lambda_{k_0}^G(\varphi_{\wti\Sigma}(v))$ of $L_{\varphi_{\wti\Sigma}(v)}$ for some $k_0\geq 1$. 
        	Recall that $W=\sigma_0^{G_\pm}(S^3,g_{S^3})$. 
        	Hence, $\varphi_{\wti\Sigma}(v)$ is area minimizing in $\wcS_{G_\pm}$, which implies that $0\leq \lambda_1(\pi(\varphi_{\wti\Sigma}(v)))=\tilde\lambda^G_1(\varphi_{\wti\Sigma}(v))=\tilde\lambda_2(\varphi_{\wti\Sigma}(v))$ by Proposition \ref{Prop: 2nd eigenfunctions}. 
        	Because $0=\tilde\lambda_{k_0}^G( \varphi_{\wti\Sigma}(v))$ for some $k_0\geq 1$, one immediately obtains $0=\tilde\lambda_{1}^G( \varphi_{\wti\Sigma}(v))$. 
        	Combined with the last statement in Proposition \ref{Prop: 2nd eigenfunctions}, we have $\Index(\varphi_{\wti\Sigma}(v))=1$ and $\Null(\varphi_{\wti\Sigma}(v))=3$ for all $v\in \mb B^3_1$. 
        This shows 
        $\varphi_{\wti\Sigma}(\mb B^3)\subset \mc G$, and the local charts $\varphi_{\wti\Sigma}\llcorner\mb B^3$ give a smooth $3$-manifold structure of $\mc G$. 
    \end{proof}

    Next, let $h$ be a complete Riemannian metric on $\mc G$, and $d_h$ be the distance function under $h$. 
    By the compactness of $\wti{\mc G}$, we can take $\bar\rho>0$ so that the $h$-geodesic ball $B^h_\rho({\wti\Sigma})$ of radius $\rho\in (0,\bar\rho]$ at ${\wti\Sigma}\in\wti{\mc G}$ is geodesically convex. 
    Additionally, as in \cite{ambrozio2024rigidity}*{(7)}, we can assume 
    \begin{align}\label{Eq: Step 4 metric equivalent}
        c\mf F({\wti\Sigma},{\wti\Sigma}')\leq d_h({\wti\Sigma},{\wti\Sigma}')\leq \mf F({\wti\Sigma},{\wti\Sigma}')\quad \mbox{for all ${\wti\Sigma}\in\wti{\mc G}$ and ${\wti\Sigma}'\in B^h_{\bar \rho}({\wti\Sigma})$}
    \end{align}
    for some $0<c<1/2$. 

    Take $0<\rho<\bar\rho/3$. 
    Then since $\lim_{i\to\infty} \sup_{x\in\wti Y_i}\mf F(|\wti\Psi_i(x)|,\wti{\mc G}) = 0$, one can use the proof of \cite{ambrozio2024rigidity}*{Proposition 6.6} to show that for each $i$ sufficiently large, 
    \begin{align}\label{Eq: Step 4 close to image}
        \sup_{x\in\wti Y_i} \mf F(\wti\Psi_i(x), \wti{\mc G}) < \rho/3.
    \end{align}
    Consider the barycentric subdivision $\wti Y_i(l_i)$ of $\wti Y_i$ so that 
    \begin{align}\label{Eq: Step 4 subdivision}
        \mf F(\wti\Psi_i(x), \wti\Psi_i(y))<\rho/3
    \end{align}
    for every $3$-cell $\alpha\in \wti{Y_i}(l_i)$ and $x,y\in\alpha$. 
    Then we define 
    \begin{align}\label{Eq: Step 4 define on vertex}
        \Psi_i(x) := {\wti\Sigma}_x \in \wti{\mc G} \quad\mbox{with $\mf F(\wti\Psi_i(x),{\wti\Sigma}_x)<\rho/3$} \qquad \mbox{for any vertex $x\in \wti Y_i(l_i)$.}
    \end{align}
    
    For any $1$-cell $[x,y]\in \wti Y_i(l_i)$, we have $d_h({\wti\Sigma}_x,{\wti\Sigma}_y)\leq \mf F({\wti\Sigma}_x,{\wti\Sigma}_y)<\rho$ by \eqref{Eq: Step 4 metric equivalent}\eqref{Eq: Step 4 subdivision}\eqref{Eq: Step 4 define on vertex}. 
    Using the geodesic convexity of $B^h_\rho({\wti\Sigma}_x)$ and $B^h_\rho({\wti\Sigma}_y)$ in $\mc G$, we can extend $\Psi_i$ to $[x,y]$ by the geodesic curve connecting ${\wti\Sigma}_x$ and ${\wti\Sigma}_y$, which gives $\Psi_i:[x,y]\to B^h_\rho({\wti\Sigma}_x)\cap B^h_\rho({\wti\Sigma}_y)$. 
    In addition, by \eqref{Eq: Step 4 metric equivalent}, 
    \begin{align*}
        \mf F(\Psi_i(z),\wti\Psi_i(z)) \leq \mf F(\Psi_i(z),{\wti\Sigma}_x) + \mf F({\wti\Sigma}_x,\wti\Psi_i(x)) + \mf F(\wti\Psi_i(x), \wti\Psi_i(z)) 
        \leq c^{-1}\rho + 2\rho/3
    \end{align*}
    for all $z\in [x,y]$. 

    Next, for any $2$-cell $\sigma_2\in \wti Y_i(l_i)$ with vertices $\{x_i\}_{i=1}^3\subset\sigma_2$, we have $\Psi_i(\bd\sigma_2)\subset B_{2\rho}^h({\wti\Sigma}_{x_i})$, $1\leq i\leq 3$, by the previous construction. 
    Then, we take any ${\wti\Sigma}\in (\cap_{i=1}^3 B_{2\rho}^h({\wti\Sigma}_{x_i})) \setminus \Psi_i(\bd\sigma_2)$, and define $\Psi_i(x_0)={\wti\Sigma}$ for the barycenter $x_0$ of $\sigma_2$. 
    Using the geodesic convexity, we can extend $\Psi_i$ on the segment from $x_0$ to $z\in\bd\sigma_2$ by the geodesic connecting ${\wti\Sigma}$ to $\Psi_i(z)$, which gives $\Psi_i:\sigma_2\to \cap_{i=1}^3 B^h_{2\rho}(\wti\Sigma_{x_i})$. 
    Similarly, we have $\mf F(\Psi_i(z),\wti\Psi_i(z)) \leq 2c^{-1}\rho + 2\rho/3$ for all $z\in\sigma_2$. 
    
    To proceed, take any $3$-cell $\sigma_3\in\wti Y_i(l_i)$ with any vertex $x\in\sigma_3$. 
    By previous constructions, we know $\Psi_i(\bd\sigma_3)\subset B^h_{3\rho}({\wti\Sigma}_x)$. 
    Then we can similarly define $\Psi_i$ at the barycenter $x_0\in\sigma_3$ to be some ${\wti\Sigma}\in B^h_{3\rho}({\wti\Sigma}_x)\setminus \Psi_i(\bd\sigma_3)$, and extend $\Psi_i$ on the segment from $x_0$ to $z\in\bd\sigma_3$ by the geodesic connecting ${\wti\Sigma}$ to $\Psi_i(z)$. 
    Note that the extended $\Psi_i$ also satisfies $\mf F(\Psi_i(z),\wti\Psi_i(z)) \leq 3c^{-1}\rho + 2\rho/3$ for all $z\in\sigma_3$. 

    Hence, we have a sequence of continuous maps $\Psi_i: \wti Y_i\to \mc G$, ($i\in\mb N$ by relabeling), with $\mf F(\Psi_i(z),\wti\Psi_i(z)) \leq 3c^{-1}\rho + 2\rho/3$ for all $z\in\wti Y_i$. 
    After taking a sequence of $\rho\to 0$ and applying a diagonal argument, we have $\lim_{i\to\infty} \sup_{x\in\wti Y_i}\mf F(|\Psi_i(x)|,\wti{\mc G}) = 0$ and $\{\Psi_i\}_{i\in\mb N}\subset\mc P^{G_\pm}_3$ (\cite{ambrozio2024rigidity}*{Lemma 2.7}).

    \medskip
    {\noindent \bf Step 5.} \textit{Up to a subsequence, $\Psi_i(\wti Y_i)\subset \mc G$ is diffeomorphic to $\mb {RP}^3$ and gives a $G_\pm$-Zoll family of minimal spheres.}

    Without loss of generality, we can assume $\wti Y_i$ is connected. 
    Denote by $Z_i:=\Psi_i(\wti Y_i)$, and by $Z\subset \mc G$ any open neighborhood of $Z_i$. 
    Since $\Psi_i$ is a $(G_\pm, 3)$-sweepout, we have non-trivial maps
    \[\mb Z_2\cong H^3(\mc Z_2(S^3;\mb Z_2);\mb Z_2) \stackrel{i^*} {\longrightarrow} H^3(\mc G;\mb Z_2) \stackrel{\iota^*} {\longrightarrow} H^3(Z;\mb Z_2) \stackrel{\iota_i^*} {\longrightarrow} H^3(Z_i;\mb Z_2) \stackrel{\Psi_i^*} {\longrightarrow} H^3(\wti Y_i;\mb Z_2),\]
    where $\iota_i,\iota,$ and $i$ are natural inclusions. 
    In particular, $H^3(Z;\mb Z_2)\neq 0$, which implies $Z$ is a closed $3$-manifold. 
    Hence, up to a subsequence, we may assume $Z_i= Z$ for all $i\in\mb N$, where $Z$ is a connected component of $\mc G$. 
    Then, consider $Z':=\{\mbox{oriented sphere } {\wti\Sigma} \in Z\}$, which is a double cover of $Z$.  
    Since $Z$ forms a sweepout, there is a curve in $Z'$ from an oriented sphere ${\wti\Sigma}$ to the same sphere with the opposite orientation, which implies that $Z'$ is connected. 

    Let $\mc PZ':=\{({\wti\Sigma},p): {\wti\Sigma}\in Z', p\in {\wti\Sigma}\}$ be the space of pointed spheres, and $\wti {G}_2(S^3)\cong S^3\times S^2$ be the set of oriented tangent $2$-planes of $S^3$. 
    Define the continuous map $T:\mc PZ'\to \wti {G}_2(S^3)$ by $T({\wti\Sigma},p):=T_p{\wti\Sigma}$. 
    Then by the proof in \cite{ambrozio2024rigidity}*{Section 6.4}, we know $\mc PZ'$ is a compact $5$-manifold, and $T$ is a homeomorphism. 
    Hence, for any $2$-plane $P$ in $T_pS^3$, there exists a unique ${\wti\Sigma}\in Z$ with $p\in{\wti\Sigma}$ so that $T_p{\wti\Sigma}=P$. 
    Moreover, the proof in \cite{ambrozio2024rigidity}*{Section 6.4} also implies that $Z'$ is simply connected. 
    Thus, $\pi_1(Z)=\mb Z_2$ and $Z\cong RP^3$ by geometrization. 
    Together, this shows that $Z$ is a $G_\pm$-Zoll family of minimal $2$-spheres for $(S^3,g_{S^3})$. 
\end{proof}

\section{Weighted Crofton formula under the surface Zoll metric}

In this section, we show a weighted Crofton formula in $RP^3$ with a surface Zoll metric $g_{RP^3}$. 
Let $\{\Sigma_a\}_{a\in \mb {RP}^3}$ be the Zoll family of minimal $RP^2$ in $(RP^3, g_{RP^3})$ so that 
\begin{itemize}
	\item[(i)] $a\in \mb {RP}^3 \mapsto \Sigma_a\subset RP^3$ is a $C^1$ map into the space of $C^{3,\alpha}$ embeddings of $RP^2$ into $RP^3$;
	\item[(ii)] for any $p\in RP^3$ and $2$-dimensional subspace $P\subset T_p(RP^3)$, there is a unique $a\in \mb {RP}^3$ satisfying $p\in\Sigma_a$ and $T_p\Sigma_a=P$. 
\end{itemize}
Since every $\Sigma_a$ is minimal, the function $a\mapsto \Area(\Sigma_a)$ has zero derivative, and thus 
\begin{align}\label{Eq: constant area}
	\Area(\Sigma_a)\equiv W \qquad \forall a\in \mb {RP}^3,
\end{align}
which coincides with the projective area widths of $(RP^3, g_{RP^3})$. 
To avoid ambiguity, denote by 
\[ A:=\mb {RP}^3 \qquad {\rm and }\qquad (M,g_M):=(RP^3, g_{RP^3}) \]
the parameter space and the ambient Riemannian manifold. 
We also use the following notation:
\begin{itemize}
	\item $N\Sigma_a$: the normal bundle of $\Sigma_a$ in $M$;
	\item $\Gamma^{3,\alpha}(N\Sigma_a)$: the space of $C^{3,\alpha}$ sections of the normal bundle $N\Sigma_a$;
	\item $G_k(M)$: the Grassmannian bundle of un-oriented tangent $k$-subspaces over $M$. 
\end{itemize}

\subsection{Incidence manifolds}

Given the Zoll family $\{\Sigma_a\}_{a\in A}$, we define its incidence set by 
\begin{align}\label{Eq: incidence set}
	\mc I:=\{(a,p)\in A\times M: p\in \Sigma_a\}. 
\end{align}
We first show the $C^1$-manifold structure of $\mc I$. 

\begin{lemma}\label{Lem: C1-manifold structure of incidence}
	The incidence set $\mc I$ is a $C^1$ embedded compact hypersurface in $A\times M$. 
\end{lemma}
\begin{proof}
	By the definition of Zoll family, for any $a_0\in A$, there exists a small neighborhood $U_{a_0}\subset A$ of $a_0$ and a $C^1$-map 
	\begin{align}\label{Eq: section map}
		S_{a_0}: U_{a_0}\to \Gamma^{3,\alpha}(N\Sigma_{a_0})
	\end{align}
	so that $\Sigma_a= \exp_{\Sigma_{a_0}}^\perp(S_{a_0}(a))$ for any $a\in U_{a_0}$. 
	For any $p_0\in \Sigma_{a_0}$, there is a neighborhood $V_{p_0}\subset \Sigma_{a_0}$ so that $\Sigma_{a_0}\llcorner V_{p_0}$ admits a unit normal $\nu_{\Sigma_{a_0}}$. 
	Then, let 
	\[ E: U_{a_0}\times V_{p_0}\to M,\qquad E(a,p):=\exp_p\left(S_{a_0}(a)(p)\right),  \]
	and define 
	\begin{align}\label{Eq: local chart for incidence manifold}
		\Psi: U_{a_0}\times V_{p_0}\to A\times M, \qquad \Psi(a,p):= (a,E(a,p)) \in \mc I.
	\end{align}
	Hence, for any $(b,q)\in \Psi(U_{a_0}\times V_{p_0})$, we have $\Psi^{-1}(b,q) = (b, n_{a_0}(q)) \in U_{a_0}\times V_{p_0}$, where $n_{a_0}$ is the $g_M$-geodesic nearest projection to $\Sigma_{a_0}$ defined near $\Sigma_{a_0}$. 
	In particular, $\mc I$ is a $C^0$-manifold with local charts given in the form of $\Psi$. 
	
	Since the section map $S_{a_0}$ in \eqref{Eq: section map} is $C^1$, we see that $(a,p)\in U_{a_0}\times V_{p_0}\mapsto S_{a_0}(a)(p)\in N\Sigma_{a_0}$ is a $C^1$ map. 
	Hence, the above $E$ and $\Psi$ are both $C^1$ maps. 
	Additionally, for any $(a,p)\in U_{a_0}\times V_{p_0}$ and $(\dot a,\dot p)\in T_aU_{a_0}\times T_pV_{p_0}$, the tangent map $d\Psi$ at $(a,p)$ is given by
	\[d\Psi_{(a,p)}(\dot a, \dot p) = \left( \dot a, ~ D_aE_{(a,p)}(\dot a) + D_pE_{(a,p)}(\dot p)  \right) \in T_aA\times T_{E(a,p)}M. \]
	We claim that $d\Psi_{(a,p)}$ is injective with rank $5$. 
	Indeed, if $d\Psi_{(a,p)}(\dot a, \dot p) =0$ for some $(\dot a,\dot p)\in T_aU_{a_0}\times T_pV_{p_0}$, then we have $\dot a=0$, and thus $D_pE_{(a,p)}(\dot p)=0$. 
	Note that $E_a:=\exp_{\Sigma_{a_0}}^\perp(S_{a_0}(a)(\cdot)) : \Sigma_{a_0}\to \Sigma_a$ is a $C^{3,\alpha}$-diffeomorphism that coincides with $E(a,\cdot )$ on $V_{p_0}$. 
	Hence, $(DE_a)_p(\dot p)=D_pE_{(a,p)}(\dot p)=0$ indicates that $\dot p=0$, which shows $d\Psi_{(a,p)}$ is injective with ${\rm rank}(d\Psi_{(a,p)}) = \dim (U_{a_0}\times V_{p_0})=5$. 
	
	Therefore, by slightly shrinking $U_{a_0}$ and $V_{p_0}$, we see that $\Psi: U_{a_0}\times V_{p_0}\to A\times M$ is a $C^1$-embedding with its image lying in $\mc I$. 
	In particular, $\mc I$ admits a $C^1$-manifold structure with local charts given as $\Psi$. 
	Since $a\mapsto\Sigma_a$ is continuous, $\mc I\subset A\times M$ is closed and compact. 
\end{proof}

Given $a\in A$ and $\dot a\in T_aA$, let $a(t)\subset A$ be a smooth curve with $a(0)=a$ and $a'(0)=\dot a$. 
Then, using the section map $S_a$ in \eqref{Eq: section map} defined near $a\in A$, $\Sigma_{a(t)} = \exp_{\Sigma_a}^\perp(S_a(a(t)))$ is a variation of minimal surfaces, whose normal variation vector field on $\Sigma_a$ is given by 
\begin{align}\label{Eq: normal variation field associated with a}
	\mc V_a(\dot a):= \left.\frac{d}{dt}\right|_{t=0} \exp_{\Sigma_a}^\perp \left( S_a(a(t)) \right) =  \left(d\exp_{\Sigma_a}^\perp\right)_0\circ (dS_a)_a (\dot a) ~\in~ \Gamma^{3,\alpha}(N\Sigma_a).
\end{align}
Since $\Sigma_{a(t)}$ has mean curvature $H_{\Sigma_{a(t)}}\equiv 0$, we have $\frac{d}{dt}|_{t=0} H_{\Sigma_{a(t)}} =0$, which implies that $\mc V_a(\dot a)$ is a Jacobi field over $\Sigma_a$. 
Hence,
\begin{align}\label{Eq: normal variation field is Jacobi}
	\mc V_a : T_a A \to \ker(L_{\Sigma_a}):=\{X\in \Gamma^{3,\alpha}(N\Sigma_a): L_{\Sigma_a}X=0\}
\end{align}
is a linear map into the $3$-dimensional linear space of Jacobi fields over $\Sigma_a$ (cf. Proposition \ref{Prop: Zoll family}). 

\begin{remark}\label{Rem: isomorphism of Jacobi map}
	Note that $\mc V_a$ in \eqref{Eq: normal variation field is Jacobi} may not be an isomorphism in general. 
	This is because the parameterization of the Zoll family is not unique. Namely, for any $C^1$ bijection $\phi: \mb {RP}^3\to \mb {RP}^3$, $\{\Sigma_{\phi(a)}\}_{a\in \mb{RP}^3}$ is also a Zoll family of minimal $RP^2$, while $d\phi$ could be degenerate so that $\mc V_a$ fails to be injective. 
	Nevertheless, by the proof of Theorem \ref{main Thm: zoll metric on RP3} and Theorem \ref{Thm: rigidity}, the Zoll family always admits {\em nondegenerate} parameterizations so that 
	\[\mbox{$\mc V_a$ is an isomorphism for any $a\in A$}.\]
	Specifically, by considering the $\mb Z_2$-quotients of the objects in the proof of Theorem \ref{Thm: rigidity}, there exists a set $\mc G$ of embedded minimal $RP^2$ in $RP^3$ so that $\mc G$ admits a $C^\infty$-manifold structure with local charts given by $\Phi_\Sigma: \mc V\in\mc W\subset \ker(L_\Sigma)\mapsto \exp_{\Sigma}^\perp(\varphi_\Sigma(\mc V))\in \mc G$, where $\Sigma\in \mc G$, $\mc W$ is a small neighborhood of $0\in \ker(L_\Sigma)$, and $\varphi_\Sigma: \mc W\to \Gamma^{3,\alpha}(N\Sigma)$ satisfies $\varphi_\Sigma(0)=0$ and $(d\varphi_\Sigma)_0=id$. 
	Additionally, $\mc G$ has a component $Z$ that is diffeomorphic to $\mb {RP}^3=A$ through a map $\iota$ (cf. {\bf Step 5} in Theorem \ref{Thm: rigidity}). 
	Then, 
	$\{\Sigma_a:=\iota(a)\}_{a\in A}=Z$ is the desired parameterization of the Zoll family since $S_a=(\exp_{\iota(a)}^\perp)^{-1}\circ\iota=\varphi_{\iota(a)}\circ \Phi_{\iota(a)}^{-1}\circ \iota$ has nondegenerate tangent map. 
\end{remark}

In the rest of this section, we always assume that $\mc V_a$ is an isomorphism for all $a\in A$ by Remark \ref{Rem: isomorphism of Jacobi map}. 
Using $\mc V_a$, we can characterize the tangent space of the incidence manifold $\mc I$. 

\begin{lemma}\label{Lem: tangent space of incidence manifold}
	Given $(a,p)\in \mc I$, the tangent space of $\mc I$ at $(a,p)$ is given by 
	\[ T_{(a,p)}\mc I =\left\{ (\dot a, v)\in T_aA\times T_pM : v^\perp_{a} = \mc V_a(\dot a)(p)   \right\} ,\]
	where $v^\perp_a$ is the normal component of $v$ with respect to $\Sigma_a$.
\end{lemma}
\begin{proof}
	Let $(a(t),p(t))$ be any curve in $\mc I$ with $(a(0),p(0))=(a,p)$ and $(\dot a,v):=(a'(0),p'(0))\in T_aA\times T_pM$. 
	Denote by $p_a(t):= n_a(p(t))$ the geodesic nearest projection of $p(t)$ to $\Sigma_a$. 
	Then, 
	\begin{align}\label{Eq: parameterization of p}
		p(t) = \exp_{\Sigma_a}^\perp \big( S_a(a(t)) \left( p_a(t) \right) \big),
	\end{align}
	where $S_a$ is the section map in \eqref{Eq: section map} defined near $a\in A$. 
	Since $\exp_{\Sigma_a}^\perp(S_a(a))=\Sigma_a$, we see that $S_a(a(0))=S_a(a)$ is the zero section of $N\Sigma_a$. 
	Combined with $(d\exp_{\Sigma_a}^\perp)_0=id$, it follows from the chain rule that 
	\begin{align}\label{Eq: derivative of p}
		v=p'(0)=\mc V_a(\dot a) (p) + p_a'(0).
	\end{align} 
	Hence, $v^\perp_a=\mc V_a(\dot a) (p)$ since $p_a'(0)\in T_p\Sigma_a$. 
	
	Conversely, for any $(\dot a, v)\in T_aA\times T_pM $ with $ v^\perp_{a} = \mc V_a(\dot a)(p)$, we can take a curve $p_a(t)$ on $\Sigma_a$ with $p_a(0)=p$ and $p_a'(0)=v-v^\perp_{a}$. 
	Then, for any curve $a(t)\subset A$ with $a(0)=a$ and $a'(0)=\dot a$, we can define $p(t):=\exp_{\Sigma_a}^\perp ( S_a(a(t)) \left( p_a(t) \right) )$ so that $p'(0)=v$. 
	This proves the lemma. 
\end{proof}

Next, we show that $\mc I$ is $C^1$-diffeomorphic to $G_1(M)$. 
\begin{lemma}\label{Lem: C1 diffeomorphism}
	Let $\mc T: \mc I \to G_1(M)$ be defined by 
	\[ \mc T (a,p) := (p,(T_p\Sigma_a)^\perp) = (p,N_p\Sigma_a) = (p, [\nu_{\Sigma_a}(p)]), \qquad \forall (a,p)\in \mc I ,\]
	where $\nu_{\Sigma_a}(p)$ is a unit normal of $\Sigma_a$ at $p$, and $[\nu_{\Sigma_a}(p)]:=\{s\nu_{\Sigma_a}(p) : s\in \mb R\}$ denotes the un-oriented normal line. 
	Then, $\mc T$ is a $C^1$-diffeomorphism. 
\end{lemma}

\begin{proof}
	Since $\{\Sigma_a\}_{a\in A}$ is a Zoll family, one easily verifies that $\mc T$ is $C^1$ and is a bijection. 
	Combined with the compactness of $\mc I$, we conclude that $\mc T$ is a homeomorphism (cf. \cite{ambrozio2024rigidity}*{Proof of Theorem 6.1}). 
	Hence, by the inverse function theorem, it is sufficient to show that the tangent map $d\mc T$ is invertible at any $(a,p)\in \mc I$. 
	
	Given any $(\dot a, v)\in T_{(a,p)}\mc I$, let $(a(t), p(t))$ be the curve in $\mc I$ with $(a(0),p(0))=(a,p)$ and $(a'(0), p'(0))=(\dot a, v)$. 
	By the proof of Lemma \ref{Lem: C1-manifold structure of incidence} and \eqref{Eq: parameterization of p}, we can locally choose the unit normals in a $C^1$-manner so that 
	\[\nu: (b,q)\in U_a\times V_p\mapsto \nu_{\Sigma_b}(\exp_{\Sigma_a}^\perp(S_a(b)(q)))\in TM\] 
	is a $C^1$-map, where $U_a$ and $V_p$ are small neighborhoods of $a\in A$ and $p\in \Sigma_a$ respectively. 
	Let
	\begin{itemize}
		\item $\nu_t:=\nu(a(t),\cdot): V_p \to TM$;
		\item $p_a(t):=n_a(p(t))$ be the nearest projection of $p(t)$ to $\Sigma_a$;
		\item $F: (t,q)\in (-\epsilon,\epsilon)\times V_p \mapsto \exp_{\Sigma_a}^\perp (S_a(a(t))(q))\in  M$, and $F_t:=F(t,\cdot)$.
	\end{itemize}
	Therefore, $\mc T(a(t),p(t))= \big( p(t), [\nu_t(p_a(t))] \big)$. 
	
	 Suppose $d\mc T_{(a,p)}(\dot a, v)=0$. 
	 Then, 
	 \[p'(0)=0\qquad{\rm and}\qquad \left.\frac{D}{dt}\right|_{t=0} \nu_t(p_a(t))=0. \] 
	 Firstly, combined with the computation in \eqref{Eq: derivative of p}, we have 
	 \[ 0=p'(0)=v=\mc V_a(\dot a)(p) + p_a'(0), \]
	 where $\mc V_a: T_aA\to \ker(L_{\Sigma_a})$ is given by \eqref{Eq: normal variation field associated with a}\eqref{Eq: normal variation field is Jacobi}. 
	 Since $\mc V_a(\dot a)(p)\in N_p\Sigma_a$ and $p_a'(0)\in T_p\Sigma_a$, we see that
	 \begin{align}\label{Eq: vanish derivative of projection pa}
	 	\mc V_a(\dot a)(p) = 0 \qquad {\rm and}\qquad p_a'(0)=0. 
	 \end{align}
	 In particular, in the neighborhood of $p$ in $\Sigma_a$, the function 
	 \begin{align}\label{Eq: local Jacobi function}
	 	u:= \langle \mc V_a(\dot a) , \nu_0\rangle 
	 \end{align}
	 satisfies $u(p) = 0$.
	 
	 Next, by the chain rule and $p_a'(0)=0$, we have 
	 \begin{align}\label{Eq: vanish derivative of unit normal}
	 	0=\left.\frac{D}{dt}\right|_{t=0} \nu_t(p_a(t))= \left.\frac{D}{dt}\right|_{t=0}\nu_t(p) + \nabla^M_{p_a'(0) } \nu_0 = \left.\frac{D}{dt}\right|_{t=0}\nu_t(p) .
	 \end{align}
	 Take any tangent vector field $X$ on $\Sigma_a$ near $p$, and define $X_t:= dF_t (X)$ to be the tangent field on $\Sigma_{a(t)}$. 
	 Since $\nu_t=\nu_{\Sigma_{a(t)}}(F_t(\cdot))$ is a local unit normal of $\Sigma_{a(t)}$, we have that 
	 \[ |\nu_t|\equiv 1 \qquad {\rm and}\qquad \langle \nu_t, X_t\rangle \equiv 0 .\]
	 Therefore, $\langle \left.\frac{D}{dt}\right|_{t=0}\nu_t(p), \nu_t(p)\rangle= 0$ and 
	 \begin{align*}
	 	\left\langle \left.\frac{D}{dt}\right|_{t=0}\nu_t, X \right\rangle 
	 	&= - \left \langle \nu_0, \left.\frac{D}{dt}\right|_{t=0} X_t  \right\rangle 
	 	= -\left\langle \nu_0,  \nabla^M_{\mc V_a(\dot a)} dF_t(X)  \right\rangle
	 	\\
	 	&= -\left\langle \nu_0,  \nabla^M_{dF_0(X)} \mc V_a(\dot a)  \right\rangle
	 	\\
	 	&= -\left\langle \nu_0,  \nabla^M_{X} (u\cdot \nu_0)  \right\rangle 
	 	\\
	 	&= -\langle \nabla^{\Sigma_a} u , X\rangle - u\cdot \langle \nu_0,\nabla^M_X\nu_0\rangle
	 	\\&= -\langle \nabla^{\Sigma_a} u , X\rangle,
	 \end{align*}
	 where $u$ is defined in \eqref{Eq: local Jacobi function}, and $X(|\nu_0|^2) = 0$ is used in the last equality. 
	 Since $X(p)\in T_p\Sigma_a$ is arbitrary, we see from \eqref{Eq: vanish derivative of unit normal} that $\nabla^{\Sigma_a} u(p)=0$, i.e. $p$ is a critical point of the function $u$. 
	 
	 Recall that $u(p)=0$ and $\nabla^{\Sigma_a} u(p)=0$. 
	 Therefore, after lifting the Jacobi field $\mc V_a(\dot a)$ of $\Sigma_a$ to the Jacobi field $\wti{\mc V}$ of the minimal $2$-sphere $\wti\Sigma_a:=\pi^{-1}(\Sigma_a)$ in $(S^3,g_{S^3})$, we see that the Jacobi function $\tilde u:= \langle \wti{\mc V} ,\tilde\nu\rangle$ on $\wti\Sigma_a$ is the second eigenfunction of $L_{\wti\Sigma_a}$ (cf. Proposition \ref{Prop: Zoll family}) and has critical points $\pi^{-1}(p)$ on its nodal set $\{x\in\wti\Sigma_a: \tilde u(x)=0\}$, where $\tilde \nu$ is a global unit normal of $\wti\Sigma_a$. 
	 By Cheng \cite{cheng1976eigenfunctions}, the eigenfunction $\tilde u$ must be the constant function $0$, which implies that $\mc V_a(\dot a)=0\in \ker(L_{\Sigma_a})$. 
	 Since $\mc V_a$ is an isomorphism by our choice of the parameterization (Remark \ref{Rem: isomorphism of Jacobi map}), we have $\dot a =0$. 
	 
	 Together, we conclude that $d\mc T_{(a,p)}$ is injective, which must be an isomorphism since $\dim(\mc I) = 5 = \dim(G_1(M))$. 
	 Then, the desired result follows from the inverse function theorem. 
\end{proof}

\subsection{Weighted Crofton formula}

Take any Riemannian metric $g_A$ on the parameter space $A=\mb {RP}^3$. We then define a metric $g_{\mc I}$ on the incidence manifold $\mc I$ by 
\begin{align}\label{Eq: metric on incidence manifold}
	g_{\mc I} \big( (\dot a_1, v_1), (\dot a_2, v_2)  \big) := g_A\big(\dot a_1, \dot a_2\big) + g_M\big( v_1-\mc V_a(\dot a_1)(p), v_2-\mc V_a(\dot a_2)(p) \big)
\end{align}
for any $(\dot a_1, v_1), (\dot a_2, v_2)\in T_{(a,p)}\mc I$. 
Intuitively, $g_{\mc I}|_{(a,p)}=g_A|_{a}+(g_{M}\llcorner \Sigma_a) |_{p}$ by Lemma \ref{Lem: tangent space of incidence manifold}. 
Hence, for any continuous function $f\in C^0(\mc I)$, 
\begin{align}\label{Eq: dV_I = dV_A dV_Sigma}
	\int_{\mc I} f(a,p) ~d\vol_{\mc I}(a,p) = \int_{A} \int_{\Sigma_a} f(a,p) ~d\vol_{\Sigma_a}(p) ~d\vol_A(a), 
\end{align}
where $d\vol_{\mc I}$, $d\vol_{\Sigma_a}$, and $d\vol_{A}$ denote the Riemannian densities (cf. \cite{lee2013introduction}*{P.427 and Proposition 16.45}) induced by the metrics $g_{\mc I}$, $g_{\Sigma_a}:=g_M\llcorner\Sigma_a$, and $g_A$ respectively.

\medskip
Meanwhile, for any $p\in M$, the fiber $G_1(T_pM)$ of the Grassmannian bundle $G_1(M)$ is identified with $S_pM/(v\sim -v)$, where 
\[S_pM:=\{v\in T_pM: |v|=1\}\] 
is the unit tangent sphere. 
Hence, the metric $g_M|_p$ induces a round metric $g_{S_pM}:=\frac{1}{2\pi} (g_M|_p) \llcorner S_pM$  on $S_pM$ with $2= \Area_{g_{S_pM}}(S_pM)= \Area_{(g_M|_p)\llcorner S_pM}(S_pM)/(2\pi)$, which further induces 
\begin{align}\label{Eq: round metric on fiber G_1}
	\mbox{a round metric $g_{G_1(T_pM)}$ on $G_1(T_pM)=S_pM/(v\sim -v)\cong \mb {RP}^2$}
\end{align} 
so that the quotient map is a local isometry. 
Hence, $\Area_{g_{G_1(T_pM)}}(G_1(T_pM))=1$, and we denote by $\omega_p$ the Riemannian density (cf. \cite{lee2013introduction}*{P.427 and Proposition 16.45}) on $G_1(T_pM)$ induced by the round metric $g_{G_1(T_pM)}$. 

Next, we define a functional $I: C^0(G_1(M))\to \mb R$ by 
\begin{align}\label{Eq: positive linear functional}
	I(f) := \frac{1}{\Vol(M)} \int_M \int_{G_1(T_pM)} f(p,[v]) ~d\omega_p([v])~d\vol_M(p),
\end{align} 
where $\Vol(M)$ is the volume of $M$ under the metric $g_M$, and $d\vol_M$ is the Riemannian density on $M$ induced by $g_M$ (cf. \cite{lee2013introduction}*{Proposition 16.45}). 
It is clear that $I$ is a positive linear functional on $C^0(G_1(M))$ and $G_1(M)$ is a compact manifold. 
Hence, by Riesz-Markov's representation theorem, there exists a positive Borel measure $\mu_{G_1(M)}$ so that 
\begin{align}\label{Eq: probability measure on G1(M)}
	I(f)=\int_{G_1(M)} f ~d\mu_{G_1(M)} \qquad\forall f\in C^0(G_1(M)),
\end{align}
and thus $\mu_{G_1(M)}(G_1(M))=1$.

\medskip
Using the $C^1$-diffeomorphism $\mc T$, we can pull-back the measure $\mu_{G_1(M)}$ to a normalized measure on $\mc I$:
\begin{align}\label{Eq: pull-back probability measure on I}
	\mu_{\mc I}:= \mc T^*\mu_{G_1(M)}. 
\end{align}
The following lemma shows that $\mu_{\mc I}$ has continuous positive density with respect to $\vol_{\mc I}$. 

\begin{lemma}\label{Lem: measure derivative}
	There exists a continuous positive function $\Theta$ on $\mc I$ so that 
	\[ \int_{\mc I} f~d\mu_{\mc I} = \int_{\mc I} f\cdot \Theta ~d\vol_{\mc I},\qquad\forall f\in C^0(\mc I). \]
	In particular, $\mu_{\mc I}$ is absolutely continuous with respect to the measure $\vol_{\mc I}$, and $\Theta$ is the associated Radon-Nikodym derivative.
\end{lemma}
\begin{remark}\label{Rem: density and volume form}
	Since $d\vol_{\mc I}$ is a Riemannian density on $\mc I$ in the sense of \cite{lee2013introduction}*{Proposition 16.45}, the above lemma indicates that $d\mu_{\mc I}$ and $d\mu_{G_1(M)}$ are densities on $\mc I$ and $G_1(M)$ respectively in the sense of \cite{lee2013introduction}*{P. 428-432}. 
\end{remark}

\begin{proof}
	By Lemma \ref{Lem: C1 diffeomorphism}, we can define the metric $g_{G_1(M)}:= (\mc T^{-1})^* g_{\mc I}$ on $G_1(M)$ with induced volume measure $\vol_{G_1(M)}$. 
	Hence, it is equivalent to show that $d\mu_{G_1(M)}=\theta \cdot d\vol_{G_1(M)}$ for some continuous function $\theta$ on $G_1(M)$. 
	
	Using the bundle structure of $G_1(M)$, we can take $x=(x_1,x_2,x_3)$ and $y=(y_1,y_2)$ as the local coordinates near $p_0\in M=RP^3$ and $[v_0]\in G_1(T_{p_0}M)\cong \mb {RP}^2$ respectively so that the submersion $G_1(M)\to M$ is expressed by $(x_1,x_2,x_3,y_1,y_2)\mapsto (x_1,x_2,x_3)$ in these coordinates. 
	Then, by the constructions, we can write the Riemannian density in these coordinates by 
	\[ d\omega_p= h_{x}(y)|dy_1dy_2|,\quad d\vol_M= H(x)|dx_1dx_2dx_3|, \quad d\vol_{G_1(M)}= G(x,y)|dx_1 dx_2 dx_3 dy_1 dy_2| \]
	for some continuous positive functions $h$, $H$, and $G$, where $|\cdot |$ denotes the density induced by the volume form (cf. \cite{lee2013introduction}*{Proposition 16.35(c)}). 
	Combined with \eqref{Eq: positive linear functional} and \eqref{Eq: probability measure on G1(M)}, we conclude that for any Borel set $U'$ in this local coordinate neighborhood, $\mu_{G_1(M)}(U')=0$ if and only if the image of $U'$ under the coordinate map is a Lebesgue null set, which is also equivalent to $\vol_{G_1(M)}(U')=0$. 
	Since the local chart exists everywhere and $G_1(M)$ is compact, the measure $\mu_{G_1(M)}$ is absolutely continuous with respect to the volume measure $\vol_{G_1(M)}$. 
	
	Let the non-negative function $\theta$ be the Radon-Nikodym derivative associated with $\mu_{G_1(M)} \ll \vol_{G_1(M)}$. 
	Then, by \eqref{Eq: positive linear functional} and \eqref{Eq: probability measure on G1(M)}, $\theta$ is expressed by $h_x(y) H(x)/(\Vol(M)\cdot G(x,y))$ in the above local coordinates, which implies that $\theta$ is positive and continuous. 
	Using the $C^1$-diffeomorphism $\mc T$ (cf. Lemma \ref{Lem: C1 diffeomorphism}) and the area formula, we obtain the desired $\Theta:=\theta\circ \mc T$. 
\end{proof}

\begin{lemma}\label{Lem: measure on A}
	Let $\Theta\in C^0(\mc I)$ be the continuous positive function given in Lemma \ref{Lem: measure derivative}. 
	Define
	\begin{align}\label{Eq: positive function m}
		m: A\to (0,\infty),\qquad  m(a) := \int_{\Sigma_a} \Theta(a,p) ~d\vol_{\Sigma_a},
	\end{align}
	as a continuous positive function on $A$, and define a positive continuous function $J$ on $\mc I$ by
	\begin{align}\label{Eq: function J}
		J: \mc I \to (0,\infty), \qquad J(a,p):= \frac{W\Theta(a,p)}{m(a)},
	\end{align}
	where $W=\Area(\Sigma_a)$ for all $a\in A$ (cf. \eqref{Eq: constant area}). 
	Then, 
	\begin{itemize}
		\item[(i)] the density $d\mu_A:= m\cdot d\vol_A$ gives a normalized measure $\mu_A$ on $A$ (i.e. $\mu_A(A) = 1$) that is independent of the choice of $g_A$; 
		\item[(ii)] $\int_A f d\mu_A = \int_{\mc I} f\circ \pi_A ~d\mu_{\mc I}$, for any $f\in C^0(A)$;
		\item[(iii)] $J$ is independent of the choice of $g_A$, $ \frac{1}{W} \int_{\Sigma_a} J (a,p)~ d\vol_{\Sigma_a}(p) = 1$ for any $a\in A$, and 
		\begin{align}
			\int_{\mc I} f d\mu_{\mc I} = \int_A \frac{1}{W} \int_{\Sigma_a} fJ ~d\vol_{\Sigma_a}~d\mu_A, \qquad\forall f\in C^0(\mc I),
		\end{align} 
	\end{itemize}
	where $\pi_A: (a,p)\in\mc I\mapsto a\in A$ is the natural projection.
\end{lemma}

\begin{proof}
	Given any $f\in C^0(A)$, it follows from \eqref{Eq: positive function m}, \eqref{Eq: dV_I = dV_A dV_Sigma}, and Lemma \ref{Lem: measure derivative} that
	\begin{align*}
		\int_Af(a) d\mu_A(a) &= \int_A  \int_{\Sigma_a} f\circ \pi_A (a,p)\cdot \Theta(a,p) ~d\vol_{\Sigma_a}  d\vol_A 
		\\
		&= \int_{\mc I} (f\circ \pi_A)\cdot \Theta ~d\vol_{\mc I}=\int_{\mc I}f\circ \pi_A ~d\mu_{\mc I}, 
	\end{align*}
	which proves (ii). 
	By taking $f\equiv 1$, we obtain $\mu_A(A)=\mu_{\mc I}(\mc I) = \mu_{G_1(M)}(G_1(M))=1$. 
	
	Let $g_A'$ be another Riemannian metric on $A$, which gives an associated metric $g_{\mc I}'$ on $\mc I$ defined similarly to \eqref{Eq: metric on incidence manifold}. 
	By the proof of Lemma \ref{Lem: measure derivative}, we have $d\vol_{A} = \rho \cdot  d\vol_A'$ and $d\vol_{\mc I}=(\rho\circ \pi_A) \cdot d\vol_{\mc I}'$ for some positive continuous function $\rho$ on $A$, where $d\vol_{A}'$ and $d\vol_{\mc I}'$ are the Riemannian densities induced by $g_A'$ and $g_{\mc I}'$ respectively. 
	Hence, $d\mu_\mc I = \Theta d\vol_\mc I = \Theta\cdot(\rho\circ \pi_A ) d\vol_{\mc I}'$, and $m'(a)= \int_{\Sigma_a} \Theta\cdot(\rho\circ \pi_A ) d\vol_{\Sigma_a} = \rho(a)\cdot  m(a)$, which implies that $\mu_A':=m'(a)d\vol_A'(a)= m(a)d\vol_A(a)=\mu_A$ is independent of the choice of $g_A$. 
	Similarly, one easily shows that $J$ is independent of $g_A$. 
	
	Moreover, it is clear that $ \frac{1}{W} \int_{\Sigma_a} J (a,p)~ d\vol_{\Sigma_a}(p) = 1$. 
	By \eqref{Eq: dV_I = dV_A dV_Sigma} and Lemma \ref{Lem: measure derivative}, 
	\begin{align*}
		\int_{\mc I} f~d\mu_{\mc I} &= \int_{\mc I} f\cdot \Theta d\vol_{\mc I} = \int_A\int_{\Sigma_a} f\cdot \Theta ~d\vol_{\Sigma_a}~d\vol_A
		\\
		&= \int_A \frac{m(a)}{W} \left(\int_{\Sigma_a} f J ~d\vol_{\Sigma_a}\right)~d\vol_A
		\\
		&= \int_A \frac{1}{W} \left(\int_{\Sigma_a} f J ~d\vol_{\Sigma_a}\right)~d\mu_A
	\end{align*}
	for any $f\in C^0(\mc I)$. 
	This shows (iii). 
\end{proof}

We are now ready to show the following weighted Crofton formula for $C^1$-embedded compact curves on $M=RP^3$ with a surface Zoll metric $g_M$.

\begin{theorem}\label{Thm: weighted Crofton formula}
	Using the above notation, 
	\begin{align}\label{Eq: weighted Crofton formula}
		\int_A \Big( \sum_{p\in \gamma\cap \Sigma_a} J(a,p)\cdot f(p) \Big) ~d\mu_A(a) = \frac{W}{2\Vol(M)} \int_\gamma f(p) ~dL_\gamma(p),
	\end{align}
	where $\gamma\subset M$ is any $C^1$-embedded compact curve, $f$ is any continuous function on $\gamma$, and $dL_\gamma$ is the length density induced by $g_{M}\llcorner\gamma$. 
\end{theorem}

\begin{remark}\label{Rem: intrinsic and rectifiable curves}
	In the above formula, $J$ and $d\mu_A$ depend only on the surface Zoll metric $g_{M}$ and the nondegenerate parameterization of the Zoll family $\{\Sigma_a\}_{a\in A}$. 
	One can also replace the parameterized Zoll family $\{\Sigma_a\}_{a\in A}$ by the un-parameterized family of minimal $RP^2$ given by the quotients of $\wti\Sigma\in Z$ in Theorem \ref{Thm: rigidity} {\bf Step 5} so that the formula \eqref{Eq: weighted Crofton formula} can be made independent of the choice of the parameterization. 
    Additionally, by using the approximate tangent lines, the following proof of Theorem \ref{Thm: weighted Crofton formula} can also be generalized to rectifiable curves (of multiplicity one). 
\end{remark}

\begin{proof}[Proof of Theorem \ref{Thm: weighted Crofton formula}]
	Let $\Pi: G_1(M)\to M$ be the bundle projection. 
	Given the curve $\gamma\subset M$, denote by $G_1(\gamma):= \Pi^{-1}(\gamma)$ the restricted Grassmannian bundle, which is a $3$-dimensional $C^1$ submanifold of $G_1(M)$. 
	Define $\mc I_\gamma:= \mc T^{-1}G_1(\gamma)$ as the corresponding $3$-dimensional $C^1$ submanifold of $\mc I$. 
	Consider the projections 
	\[ \pi_A:(a,p)\in\mc I \mapsto a\in A \qquad{\rm and}\qquad\pi_M:(a,p)\in\mc I \mapsto p\in M .\]
	Using the local chart \eqref{Eq: local chart for incidence manifold} of $\mc I$ and the fact that $\pi_M=\Pi\circ\mc T$, $\pi_A$ and $\pi_M$ are $C^1$-submersions from $\mc I$ to $A$ and $M$ respectively. 
	Hence, the restricted projection $\pi_A^\gamma=\pi_A\llcorner \mc I_{\gamma} : \mc I_\gamma\to A$ is a $C^1$-map between two $3$-dimensional manifolds. 
	By the area formula, 
	\begin{align}\label{Eq: area formula}
		\int_A \Big(\sum_{p\in(\pi_A^\gamma)^{-1}(a)} J(a,p) f(p)\Big) ~d\mu_A(a) = \int_{\mc I_\gamma}  J\cdot (f\circ\pi_M) ~(\pi_A^\gamma)^*d\mu_A, \quad\forall f\in C^0(\gamma),
	\end{align}
	where $J$ is given by \eqref{Eq: function J}, $d\mu_A$ is the Riemannian density induced by the conformal metric $\tilde g_A:=m^{2/3} g_A$ by Lemma \ref{Lem: measure on A}(i), and $(\pi_A^\gamma)^*d\mu_A$ is the pull-back density. 
	
	To compute $J\cdot (\pi_A^\gamma)^*d\mu_A$, we fix any $(a,p)\in \mc I_\gamma$ so that $\gamma$ is transversal to $\Sigma_a$ at $p$. (Clearly, such a transversal point $(a,p)\in \mc I_\gamma$ exists almost everywhere.) 
	Take 
	\begin{itemize}
		\item $\xi_1,\xi_2\in \ker(d\pi_M):=\{(\dot a, v)\in T_{(a,p)}\mc I : d\pi_M(\dot a, v)=0\}=\{(\dot a,0)\in T_{(a,p)}\mc I: \mc V_a(\dot a)(p)=0\}$;
		\item $\tau\in 	T_{(a,p)}\mc I_\gamma$ with $\gamma'|_{p} = d\pi_M(\tau )$, where $\gamma'$ is the unit tangent vector field on $\gamma$;
		\item $e_1,e_2\in \ker(d\pi_A):=\{(\dot a, v)\in T_{(a,p)}\mc I : d\pi_A(\dot a, v)=0\}=\{(0,v)\in T_aA\times T_p\Sigma_a\}$ so that $\{E_i:=d\pi_M(e_i)\}_{i=1}^2\subset T_p\Sigma_a$ forms an orthonormal basis. 
	\end{itemize}
	Recall that $g_\mc I|_{(a,p)}= (d\pi_A)^*g_{A}|_a + d(n_{\Sigma_a}\circ\pi_M)^* g_{\Sigma_a}|_p$ by \eqref{Eq: metric on incidence manifold}, where $n_{\Sigma_a}$ is the $g_M$-geodesic nearest projection to $\Sigma_a$. 
	Thus, for any $\xi\in T_{(a,p)}\mc I$, $d\pi_A\xi \in T_aA$ can be viewed as the projection of $\xi$ to the orthogonal complement of $\ker(d\pi_A)\cong T_p\Sigma_a$ in $(T_{(a,p)}\mc I, g_{\mc I}|_{(a,p)})$. 
	Hence, 
	\begin{align*}
		J(a,p)\cdot [(\pi_A^\gamma)^*d\mu_A](\xi_1,\xi_2,\tau) &= J(a,p) \cdot d\mu_A\left( d\pi_A^\gamma\xi_1, d\pi_A^\gamma \xi_2, d\pi_A^\gamma\tau \right)
		\\
		&= J(a,p) m(a) \cdot d\vol_A\left( d\pi_A^\gamma\xi_1, d\pi_A^\gamma \xi_2, d\pi_A^\gamma\tau \right) \cdot d\vol_{\Sigma_a}(d\pi_M e_1, d\pi_M e_2)
		\\
		&= W\Theta(a,p)\cdot d\vol_{\mc I}(\xi_1,\xi_2,\tau,e_1,e_2)
		\\
		&= W  d\mu_{\mc I}(\xi_1,\xi_2,\tau,e_1,e_2),
	\end{align*}
	where we used Lemma \ref{Lem: measure on A}(i) and the fact that $d\vol_{\Sigma_a}(E_1,E_2)=1$ in the second equality, used \eqref{Eq: metric on incidence manifold}\eqref{Eq: dV_I = dV_A dV_Sigma} in the third equality, and used Lemma \ref{Lem: measure derivative} in the last equality. 
	
	Since $\pi_M=\Pi\circ \mc T$ and $\gamma$ is transversal to $\Sigma_a$, we know that $d\mc T\xi_1,d\mc T\xi_2 \in \ker(d\Pi)$ are tangent to the fiber $G_1(T_pM)$, while $\{d\mc T\tau,d\mc Te_1,d\mc Te_2\}$ are mapped to a basis of $T_pM$ by $d\Pi$. 
	Therefore, combining this with Remark \ref{Rem: density and volume form}, \eqref{Eq: pull-back probability measure on I} and \eqref{Eq: positive linear functional}, 
	\begin{align*}
		J(a,p)\cdot [(\pi_A^\gamma)^*d\mu_A](\xi_1,\xi_2,\tau) &= W\cdot d\mu_{G_1(M)}(d\mc T\xi_1, d\mc T\xi_2, d\mc T\tau, d\mc Te_1, d\mc Te_2)
		\\
		&= \frac{W}{\Vol(M)} \cdot d\omega_p(d\mc T\xi_1, d\mc T\xi_2) \cdot d\vol_M(\gamma'|_p, E_1,E_2)
		\\
		&= \frac{W}{\Vol(M)} \cdot d\omega_p(d\mc T\xi_1, d\mc T\xi_2)\cdot  \left | \left \langle \gamma'|_p, \nu_{\Sigma_a}(p) \right \rangle \right|,
	\end{align*}
	where $\nu_{\Sigma_a}(p)$ is a unit normal of $\Sigma_a$ at $p$, the fact that $\{E_1,E_2\}$ is an orthonormal basis of $T_p\Sigma_a$ is used in the last line, and we take the absolute value since the density ignores the orientation. 
	
	Now, we have shown that the density $J\cdot [(\pi_A^\gamma)^*d\mu_A]$ at any transversal point $(a,p)\in \mc I_\gamma$ coincides with the density $\frac{W}{\Vol(M)} \left | \left \langle \gamma'|_p, \nu_{\Sigma_a} \right \rangle \right|\cdot \mc T_p^*d\omega_p\cdot   dL_\gamma$ on the fiber bundle $\pi_M=\Pi\circ \mc T:\mc I_\gamma \to \gamma$, where $\mc T_p:=\mc T\llcorner \pi_M^{-1}(p)$. 
    Additionally, if $(a,p)\in \mc I_\gamma$ with $\gamma$ tangent to $\Sigma_a$ at $p$, then since $T_{(a,p)}\mc I_\gamma=\{(\dot a, v): v\in T_p\gamma, v^\perp_a=\mc V_a(\dot a)(p), \}$ and $(d\pi_A^\gamma)_{(a,p)}(\dot a, v)=\dot a$, we know $\ker((d\pi_A^\gamma)_{(a,p)})=\{(0,v):v\in T_p\gamma\cap T_p\Sigma_a\}$ has dimension $1$, which implies that $d\pi_A^\gamma$ is not full rank and  $(\pi_A^\gamma)^*d\mu_A=0=| \left \langle \gamma'|_p, \nu_{\Sigma_a}(p) \right \rangle |$. 
	Combining these with \eqref{Eq: area formula}, we see that
	\begin{align*}
		\int_A \Big(\sum_{p\in(\pi_A^\gamma)^{-1}(a)} J(a,p)f(p)\Big) ~d\mu_A(a) 
		&=\frac{W}{\Vol(M)} \int_\gamma f(p) \int_{\pi_M^{-1}(p)} \left | \left \langle \gamma'|_{p}, \nu_{\Sigma_a} \right \rangle \right| \mc T_{p}^*d\omega_{p}(a) ~dL_\gamma(p)
		\\
		&= \frac{W}{\Vol(M)} \int_\gamma f(p) \int_{G_1(T_pM)} \left | \left \langle \gamma'|_p, \nu \right \rangle \right| d\omega_p([\nu]) ~dL_\gamma(p).
	\end{align*}
	Finally, note that $\gamma'$ is a unit tangent vector field on $\gamma$. 
	Hence, under the normalized round metric $g_{G_1(T_pM)}$ in \eqref{Eq: round metric on fiber G_1}, we have 
	\begin{align*}
		\int_{G_1(T_pM)}  \left | \left \langle \gamma'(p), \nu \right \rangle \right| d\omega_p([\nu]) &= \frac{1}{2}\int_{S_pM} \left | \left \langle \gamma'(p), \nu \right \rangle \right| d\Area_{g_{S_pM}}(\nu) = \frac{1}{4\pi} \int_{\mb S^2_1} |\cos(\theta_\nu)| d\Area_{g_{\mb S^2_1}}(\nu) 
		\\
		&= \frac{1}{4\pi} \int_0^{2\pi}\int_0^\pi |\cos(\theta)|\sin(\theta) d\theta d\phi \equiv 1/2,
	\end{align*}
	which shows the desired result.
\end{proof}



The results in Theorem \ref{Main Thm: weighted Crofton formula} now follow directly from Lemma \ref{Lem: measure on A} and Theorem \ref{Thm: weighted Crofton formula}.
As a corollary, we have the following systolic inequality on $M=RP^3$ with surface Zoll metrics, which further implies an inequality that relates the volume and projective area widths. 

\begin{corollary}
    Suppose $J:\mc I\to (0,\infty)$, $W>0$, and $\mu_A$ are given as in Lemma \ref{Lem: measure on A}. 
    Let $j(a):=\min_{p\in\Sigma_a}J(a,p)$ for any $a\in \mb {RP}^3$, and let $\varsigma:=\int_{A}j(a)d\mu_A\in (0,1]$. 
    Then, the systole $\text{sys}(M,g_M)$ of $M=RP^3$ with a surface Zoll metric $g_M$ satisfies
    \[ \varsigma\cdot\frac{2\Vol(M)}{W}\leq \text{sys}(M,g_M) \leq \sqrt{\frac{\pi W}{2}}. \]
    In particular, $8\varsigma^2\cdot(\Vol(M))^2\leq \pi W^3$ with equality if and only if $g_M$ has constant sectional curvature. 
\end{corollary}
\begin{proof}
    Take any non-contractible $C^1$-embedded loop $\gamma\subset M=RP^3$. 
    Then, $[\gamma]\neq 0\in H_1(RP^3;\mb Z_2)$ and $[\Sigma_a]\neq 0\in H_2(RP^3;\mb Z_2)$ for any $a\in A$. 
    Hence, for generic $a\in A$, $\gamma$ and $\Sigma_a$ have mod-$2$ intersection number $1$, which implies $\#(\gamma\cap\Sigma_a)\geq 1$ and $\sum_{\gamma\cap\Sigma_a} J(a,p)\geq j(a)$. 
    By taking $f\equiv 1$ in Theorem \ref{Thm: weighted Crofton formula}, we have $ \frac{W}{2\Vol(M)} \text{Length}(\gamma)\geq \int_Aj(a)d\mu_A(a) = \varsigma$. 
    This shows the lower bound of the systole. 
    Finally, by \cite{ambrozio2024rigidity}*{Theorem C}, one obtains the upper bound of the systole and the last statement. 
\end{proof}

\bibliographystyle{abbrv}

\bibliography{reference.bib}   

@article{allard1972first,
    AUTHOR = {Allard, William K.},
     TITLE = {On the first variation of a varifold},
   JOURNAL = {Ann. of Math. (2)},
  FJOURNAL = {Annals of Mathematics. Second Series},
    VOLUME = {95},
      YEAR = {1972},
     PAGES = {417--491},
      ISSN = {0003-486X},
   MRCLASS = {49F20},
  MRNUMBER = {307015},
MRREVIEWER = {M.\ Klingmann},
       DOI = {10.2307/1970868},
       URL = {https://doi.org/10.2307/1970868},
}

@article{almgren1962homotopy,
    AUTHOR = {Almgren, Jr., Frederick Justin},
     TITLE = {The homotopy groups of the integral cycle groups},
   JOURNAL = {Topology},
  FJOURNAL = {Topology. An International Journal of Mathematics},
    VOLUME = {1},
      YEAR = {1962},
     PAGES = {257--299},
      ISSN = {0040-9383},
   MRCLASS = {55.45 (55.42)},
  MRNUMBER = {146835},
       DOI = {10.1016/0040-9383(62)90016-2},
       URL = {https://doi.org/10.1016/0040-9383(62)90016-2},
}

@article{ambrozio2024rigidity,
	title={Rigidity theorems for the area widths of {R}iemannian manifolds}, 
      author={Lucas Ambrozio and Fernando C. Marques and Andr{\'e} Neves},
      year={2024},
      journal={arXiv preprint, arXiv:2408.14375}
}

@article {ambrozio2025metrics,
    AUTHOR = {Ambrozio, Lucas and Marques, Fernando C. and Neves, Andr\'e},
     TITLE = {Riemannian metrics on the sphere with {Z}oll families of
              minimal hypersurfaces},
   JOURNAL = {J. Differential Geom.},
  FJOURNAL = {Journal of Differential Geometry},
    VOLUME = {130},
      YEAR = {2025},
    NUMBER = {2},
     PAGES = {269--341},
      ISSN = {0022-040X,1945-743X},
   MRCLASS = {53C42 (53C40)},
  MRNUMBER = {4905026},
MRREVIEWER = {Isabel\ M. C. Salavessa},
       DOI = {10.4310/jdg/1747156792},
       URL = {https://doi.org/10.4310/jdg/1747156792},
}

@article{ambrozio2025equivariant,
	title={Equivariant constructions of spheres with {Z}oll families of minimal spheres}, 
      author={Lucas Ambrozio and Diego Guajardo},
      year={2025},
      journal={arXiv preprint, arXiv:2501.16032v2}
}

@article {ambrozio2026spheres,
    AUTHOR = {Ambrozio, Lucas},
     TITLE = {Spheres with minimal equators},
   JOURNAL = {S\~ao Paulo J. Math. Sci.},
  FJOURNAL = {S\~ao Paulo Journal of Mathematical Sciences},
    VOLUME = {20},
      YEAR = {2026},
    NUMBER = {2},
     PAGES = {Paper No. 25, 31},
      ISSN = {1982-6907,2316-9028},
   MRCLASS = {53A10},
  MRNUMBER = {5103556},
       DOI = {10.1007/s40863-026-00548-0},
       URL = {https://doi.org/10.1007/s40863-026-00548-0},
}

@article {bray2010area,
    AUTHOR = {Bray, H. and Brendle, S. and Eichmair, M. and Neves, A.},
     TITLE = {Area-minimizing projective planes in 3-manifolds},
   JOURNAL = {Comm. Pure Appl. Math.},
  FJOURNAL = {Communications on Pure and Applied Mathematics},
    VOLUME = {63},
      YEAR = {2010},
    NUMBER = {9},
     PAGES = {1237--1247},
      ISSN = {0010-3640,1097-0312},
   MRCLASS = {53C42 (53C20)},
  MRNUMBER = {2675487},
MRREVIEWER = {St\'{e}phane\ Sabourau},
       DOI = {10.1002/cpa.20319},
       URL = {https://doi.org/10.1002/cpa.20319},
}

@book {besse1978manifolds,
    AUTHOR = {Besse, Arthur L.},
     TITLE = {Manifolds all of whose geodesics are closed},
    SERIES = {Ergebnisse der Mathematik und ihrer Grenzgebiete [Results in
              Mathematics and Related Areas]},
    VOLUME = {93},
      NOTE = {With appendices by D. B. A. Epstein, J.-P. Bourguignon, L.
              B\'erard-Bergery, M. Berger and J. L. Kazdan},
 PUBLISHER = {Springer-Verlag, Berlin-New York},
      YEAR = {1978},
     PAGES = {ix+262},
      ISBN = {3-540-08158-5},
   MRCLASS = {53C20 (53C22 58G99)},
  MRNUMBER = {496885},
MRREVIEWER = {R.\ L.\ Bishop},
}

@article {cheng1976eigenfunctions,
    AUTHOR = {Cheng, Shiu Yuen},
     TITLE = {Eigenfunctions and nodal sets},
   JOURNAL = {Comment. Math. Helv.},
  FJOURNAL = {Commentarii Mathematici Helvetici},
    VOLUME = {51},
      YEAR = {1976},
    NUMBER = {1},
     PAGES = {43--55},
      ISSN = {0010-2571,1420-8946},
   MRCLASS = {58G99 (35P15)},
  MRNUMBER = {397805},
MRREVIEWER = {Sh\^ukichi\ Tanno},
       DOI = {10.1007/BF02568142},
       URL = {https://doi.org/10.1007/BF02568142},
}

@article {chodosh2023p-widths,
    AUTHOR = {Chodosh, Otis and Mantoulidis, Christos},
     TITLE = {The {$p$}-widths of a surface},
   JOURNAL = {Publ. Math. Inst. Hautes \'Etudes Sci.},
  FJOURNAL = {Publications Math\'ematiques. Institut de Hautes \'Etudes
              Scientifiques},
    VOLUME = {137},
      YEAR = {2023},
     PAGES = {245--342},
      ISSN = {0073-8301,1618-1913},
   MRCLASS = {53C42 (49Q15 49Q20 58J50)},
  MRNUMBER = {4588597},
MRREVIEWER = {Alexis\ Michelat},
       DOI = {10.1007/s10240-023-00141-7},
       URL = {https://doi.org/10.1007/s10240-023-00141-7},
}

@article {franz2023index,
    AUTHOR = {Franz, Giada},
     TITLE = {Equivariant index bound for min-max free boundary minimal
              surfaces},
   JOURNAL = {Calc. Var. Partial Differential Equations},
  FJOURNAL = {Calculus of Variations and Partial Differential Equations},
    VOLUME = {62},
      YEAR = {2023},
    NUMBER = {7},
     PAGES = {Paper No. 201, 28},
      ISSN = {0944-2669,1432-0835},
   MRCLASS = {53C42 (49J35 49Q20 53A10)},
  MRNUMBER = {4621518},
MRREVIEWER = {Alexis\ Michelat},
       DOI = {10.1007/s00526-023-02514-6},
       URL = {https://doi.org/10.1007/s00526-023-02514-6},
}

@book{federer2014geometric,
    AUTHOR = {Federer, Herbert},
     TITLE = {Geometric measure theory},
    SERIES = {Die Grundlehren der mathematischen Wissenschaften},
    VOLUME = {153},
 PUBLISHER = {Springer-Verlag New York, Inc., New York},
      YEAR = {1969},
     PAGES = {xiv+676},
   MRCLASS = {28.80 (26.00)},
  MRNUMBER = {257325},
MRREVIEWER = {J.\ E.\ Brothers},
}

@book{simon1983lectures,
    AUTHOR = {Simon, Leon},
     TITLE = {Lectures on geometric measure theory},
    SERIES = {Proceedings of the Centre for Mathematical Analysis,
              Australian National University},
    VOLUME = {3},
 PUBLISHER = {Australian National University, Centre for Mathematical
              Analysis, Canberra},
      YEAR = {1983},
     PAGES = {vii+272},
      ISBN = {0-86784-429-9},
   MRCLASS = {49-01 (28A75 49F20)},
  MRNUMBER = {756417},
MRREVIEWER = {J.\ S.\ Joel},
}

@article {grayson1989shortening,
    AUTHOR = {Grayson, Matthew A.},
     TITLE = {Shortening embedded curves},
   JOURNAL = {Ann. of Math. (2)},
  FJOURNAL = {Annals of Mathematics. Second Series},
    VOLUME = {129},
      YEAR = {1989},
    NUMBER = {1},
     PAGES = {71--111},
      ISSN = {0003-486X,1939-8980},
   MRCLASS = {53C22 (58E10)},
  MRNUMBER = {979601},
MRREVIEWER = {Gudlaugur\ Thorbergsson},
       DOI = {10.2307/1971486},
       URL = {https://doi.org/10.2307/1971486},
}

@article {galvez2020uniquesness,
    AUTHOR = {G{\'a}lvez, Jos{\'e} A. and Mira, Pablo},
     TITLE = {Uniqueness of immersed spheres in three-manifolds},
   JOURNAL = {J. Differential Geom.},
  FJOURNAL = {Journal of Differential Geometry},
    VOLUME = {116},
      YEAR = {2020},
    NUMBER = {3},
     PAGES = {459--480},
      ISSN = {0022-040X,1945-743X},
   MRCLASS = {53A10 (53C42)},
  MRNUMBER = {4182894},
MRREVIEWER = {Martin\ L. P. Kilian},
       DOI = {10.4310/jdg/1606964415},
       URL = {https://doi.org/10.4310/jdg/1606964415},
}

@article{gromov1988dimension,
	author={Gromov, Mikhael},
	title={Dimension, non-linear spectra and width},
	journal={Geometric Aspects of Functional Analysis (1986/87), Lecture Notes in Math},
	year={1988},
	publisher={Springer Berlin Heidelberg},
	address={Berlin, Heidelberg},
	pages={132--184},
	isbn={978-3-540-39235-4}
}

@article{gromov2003isoperimetry,
    AUTHOR = {Gromov, M.},
     TITLE = {Isoperimetry of waists and concentration of maps},
   JOURNAL = {Geom. Funct. Anal.},
  FJOURNAL = {Geometric and Functional Analysis},
    VOLUME = {13},
      YEAR = {2003},
    NUMBER = {1},
     PAGES = {178--215},
      ISSN = {1016-443X,1420-8970},
   MRCLASS = {53C23},
  MRNUMBER = {1978494},
MRREVIEWER = {Igor\ Belegradek},
       DOI = {10.1007/s000390300004},
       URL = {https://doi.org/10.1007/s000390300004},
}

@article{guth2009minimax,
    AUTHOR = {Guth, Larry},
     TITLE = {Minimax problems related to cup powers and {S}teenrod squares},
   JOURNAL = {Geom. Funct. Anal.},
  FJOURNAL = {Geometric and Functional Analysis},
    VOLUME = {18},
      YEAR = {2009},
    NUMBER = {6},
     PAGES = {1917--1987},
      ISSN = {1016-443X,1420-8970},
   MRCLASS = {53C23},
  MRNUMBER = {2491695},
MRREVIEWER = {John\ F.\ Oprea},
       DOI = {10.1007/s00039-009-0710-2},
       URL = {https://doi.org/10.1007/s00039-009-0710-2},
}

@article {guillemin1976radon,
    AUTHOR = {Guillemin, Victor},
     TITLE = {The {R}adon transform on {Z}oll surfaces},
   JOURNAL = {Advances in Math.},
  FJOURNAL = {Advances in Mathematics},
    VOLUME = {22},
      YEAR = {1976},
    NUMBER = {1},
     PAGES = {85--119},
      ISSN = {0001-8708},
   MRCLASS = {58G15 (53C20)},
  MRNUMBER = {426063},
MRREVIEWER = {J.\ Eells},
       DOI = {10.1016/0001-8708(76)90139-0},
       URL = {https://doi.org/10.1016/0001-8708(76)90139-0},
}

@article{hatcher1983smale,
    AUTHOR = {Hatcher, Allen E.},
     TITLE = {A proof of the {S}male conjecture, {${\rm Diff}(S\sp{3})\simeq
              {\rm O}(4)$}},
   JOURNAL = {Ann. of Math. (2)},
  FJOURNAL = {Annals of Mathematics. Second Series},
    VOLUME = {117},
      YEAR = {1983},
    NUMBER = {3},
     PAGES = {553--607},
      ISSN = {0003-486X,1939-8980},
   MRCLASS = {57M99 (57S05)},
  MRNUMBER = {701256},
MRREVIEWER = {R.\ C.\ Kirby},
       DOI = {10.2307/2007035},
       URL = {https://doi.org/10.2307/2007035},
}

@article {irie2018density,
    AUTHOR = {Irie, Kei and Marques, Fernando C. and Neves, Andr{\'e}},
     TITLE = {Density of minimal hypersurfaces for generic metrics},
   JOURNAL = {Ann. of Math. (2)},
  FJOURNAL = {Annals of Mathematics. Second Series},
    VOLUME = {187},
      YEAR = {2018},
    NUMBER = {3},
     PAGES = {963--972},
      ISSN = {0003-486X,1939-8980},
   MRCLASS = {53C42 (49Q05)},
  MRNUMBER = {3779962},
MRREVIEWER = {Jos\'e\ Miguel\ Manzano},
       DOI = {10.4007/annals.2018.187.3.8},
       URL = {https://doi.org/10.4007/annals.2018.187.3.8},
}

@article {kac1966can,
    AUTHOR = {Kac, Mark},
     TITLE = {Can one hear the shape of a drum?},
   JOURNAL = {Amer. Math. Monthly},
  FJOURNAL = {American Mathematical Monthly},
    VOLUME = {73},
      YEAR = {1966},
    NUMBER = {4},
     PAGES = {1--23},
      ISSN = {0002-9890,1930-0972},
   MRCLASS = {57.50 (00.00)},
  MRNUMBER = {201237},
MRREVIEWER = {I.\ Stakgold},
       DOI = {10.2307/2313748},
       URL = {https://doi.org/10.2307/2313748},
}

@article{lusternic1947topoligical,
    AUTHOR = {Lyusternik, L. and Shnirelman, L.},
     TITLE = {Topological methods in variational problems and their
              application to the differential geometry of surfaces},
   JOURNAL = {Uspehi Matem. Nauk (N.S.)},
  FJOURNAL = {Uspehi Matem. Nauk (N.S.)},
    VOLUME = {2},
      YEAR = {1947},
    NUMBER = {1(17)},
     PAGES = {166--217},
   MRCLASS = {53.0X},
  MRNUMBER = {29532},
MRREVIEWER = {H.\ Busemann},
}

@book {lee2013introduction,
    AUTHOR = {Lee, John M.},
     TITLE = {Introduction to smooth manifolds},
    SERIES = {Graduate Texts in Mathematics},
    VOLUME = {218},
   EDITION = {Second},
 PUBLISHER = {Springer, New York},
      YEAR = {2013},
     PAGES = {xvi+708},
      ISBN = {978-1-4419-9981-8},
   MRCLASS = {58-01 (53-01 57-01)},
  MRNUMBER = {2954043},
}

@article{li2024lowgenus,
    AUTHOR = {Li, Xingzhe and Wang, Tongrui and Yao, Xuan},
     TITLE = {Minimal surfaces with low genus in lens spaces},
   JOURNAL = {J. Reine Angew. Math.},
  FJOURNAL = {Journal f\"ur die Reine und Angewandte Mathematik. [Crelle's
              Journal]},
    VOLUME = {828},
      YEAR = {2025},
     PAGES = {175--218},
      ISSN = {0075-4102,1435-5345},
   MRCLASS = {53A10 (49Q05 53A20)},
  MRNUMBER = {4979239},
       DOI = {10.1515/crelle-2025-0061},
       URL = {https://doi.org/10.1515/crelle-2025-0061},
}

@article {liokumovich2018weyl,
    AUTHOR = {Liokumovich, Yevgeny and Marques, Fernando C. and Neves,
              Andr\'e},
     TITLE = {Weyl law for the volume spectrum},
   JOURNAL = {Ann. of Math. (2)},
  FJOURNAL = {Annals of Mathematics. Second Series},
    VOLUME = {187},
      YEAR = {2018},
    NUMBER = {3},
     PAGES = {933--961},
      ISSN = {0003-486X,1939-8980},
   MRCLASS = {53C23 (58E05 58J50)},
  MRNUMBER = {3779961},
MRREVIEWER = {Leonid\ Friedlander},
       DOI = {10.4007/annals.2018.187.3.7},
       URL = {https://doi.org/10.4007/annals.2018.187.3.7},
}

@article {mazzucchelli2018characterization,
    AUTHOR = {Mazzucchelli, Marco and Suhr, Stefan},
     TITLE = {A characterization of {Z}oll {R}iemannian metrics on the {$2$}-sphere},
   JOURNAL = {Bull. Lond. Math. Soc.},
  FJOURNAL = {Bulletin of the London Mathematical Society},
    VOLUME = {50},
      YEAR = {2018},
    NUMBER = {6},
     PAGES = {997--1006},
      ISSN = {0024-6093,1469-2120},
   MRCLASS = {53C22 (58E10)},
  MRNUMBER = {3891938},
MRREVIEWER = {St\'ephane\ Sabourau},
       DOI = {10.1112/blms.12200},
       URL = {https://doi.org/10.1112/blms.12200},
}

@article{marques2016morse,
    AUTHOR = {Marques, Fernando C. and Neves, Andr\'e},
     TITLE = {Morse index and multiplicity of min-max minimal hypersurfaces},
   JOURNAL = {Camb. J. Math.},
  FJOURNAL = {Cambridge Journal of Mathematics},
    VOLUME = {4},
      YEAR = {2016},
    NUMBER = {4},
     PAGES = {463--511},
      ISSN = {2168-0930,2168-0949},
   MRCLASS = {49J35 (58E12)},
  MRNUMBER = {3572636},
MRREVIEWER = {Giandomenico\ Orlandi},
       DOI = {10.4310/CJM.2016.v4.n4.a2},
       URL = {https://doi.org/10.4310/CJM.2016.v4.n4.a2},
}

@article {marques2021morse,
    AUTHOR = {Marques, Fernando C. and Neves, Andr\'e},
     TITLE = {Morse index of multiplicity one min-max minimal hypersurfaces},
   JOURNAL = {Adv. Math.},
  FJOURNAL = {Advances in Mathematics},
    VOLUME = {378},
      YEAR = {2021},
     PAGES = {Paper No. 107527, 58},
      ISSN = {0001-8708,1090-2082},
   MRCLASS = {58E12},
  MRNUMBER = {4191255},
MRREVIEWER = {Futoshi\ Takahashi},
       DOI = {10.1016/j.aim.2020.107527},
       URL = {https://doi.org/10.1016/j.aim.2020.107527},
}

@article{marques2017existence,
    AUTHOR = {Marques, Fernando C. and Neves, Andr\'e},
     TITLE = {Existence of infinitely many minimal hypersurfaces in positive
              {R}icci curvature},
   JOURNAL = {Invent. Math.},
  FJOURNAL = {Inventiones Mathematicae},
    VOLUME = {209},
      YEAR = {2017},
    NUMBER = {2},
     PAGES = {577--616},
      ISSN = {0020-9910,1432-1297},
   MRCLASS = {53C42 (49Q05 53C21 58E12)},
  MRNUMBER = {3674223},
MRREVIEWER = {Martin\ Man-Chun\ Li},
       DOI = {10.1007/s00222-017-0716-6},
       URL = {https://doi.org/10.1007/s00222-017-0716-6},
}

@article {marques2019equidistribution,
    AUTHOR = {Marques, Fernando C. and Neves, Andr\'e and Song, Antoine},
     TITLE = {Equidistribution of minimal hypersurfaces for generic metrics},
   JOURNAL = {Invent. Math.},
  FJOURNAL = {Inventiones Mathematicae},
    VOLUME = {216},
      YEAR = {2019},
    NUMBER = {2},
     PAGES = {421--443},
      ISSN = {0020-9910,1432-1297},
   MRCLASS = {53C42 (49Q05 49Q20 53A10 58D17 58E12)},
  MRNUMBER = {3953507},
MRREVIEWER = {S.\ Timothy\ Swift},
       DOI = {10.1007/s00222-018-00850-5},
       URL = {https://doi.org/10.1007/s00222-018-00850-5},
}

@article{martins2026spectral,
	author = {Martins, Gustavo},
	journal = {arXiv preprint arXiv:2609.20689},
	title = {Topological and spectral rigidity of hypersurface {Z}oll manifolds},
	year = {2026}}

@article {marx-kuo2025isospectral,
    AUTHOR = {Marx-Kuo, Jared},
     TITLE = {The isospectral problem for {$p$}-widths: an application of
              {Z}oll metrics},
   JOURNAL = {C. R. Math. Acad. Sci. Paris},
  FJOURNAL = {Comptes Rendus Math\'ematique. Acad\'emie des Sciences. Paris},
    VOLUME = {363},
      YEAR = {2025},
     PAGES = {565--570},
      ISSN = {1631-073X,1778-3569},
   MRCLASS = {53C22 (58J53)},
  MRNUMBER = {4916697},
       DOI = {10.5802/crmath.708},
       URL = {https://doi.org/10.5802/crmath.708},
}

@book{pitts2014existence,
    AUTHOR = {Pitts, Jon T.},
     TITLE = {Existence and regularity of minimal surfaces on {R}iemannian
              manifolds},
    SERIES = {Mathematical Notes},
    VOLUME = {27},
 PUBLISHER = {Princeton University Press, Princeton, NJ; University of Tokyo
              Press, Tokyo},
      YEAR = {1981},
     PAGES = {iv+330},
      ISBN = {0-691-08290-1},
   MRCLASS = {49F22 (53C42)},
  MRNUMBER = {626027},
MRREVIEWER = {J.\ E.\ Brothers},
}

@article {sharp2017compactness,
    AUTHOR = {Sharp, Ben},
     TITLE = {Compactness of minimal hypersurfaces with bounded index},
   JOURNAL = {J. Differential Geom.},
  FJOURNAL = {Journal of Differential Geometry},
    VOLUME = {106},
      YEAR = {2017},
    NUMBER = {2},
     PAGES = {317--339},
      ISSN = {0022-040X,1945-743X},
   MRCLASS = {53C42 (53C21)},
  MRNUMBER = {3662994},
MRREVIEWER = {Luciano\ Mari},
       DOI = {10.4310/jdg/1497405628},
       URL = {https://doi.org/10.4310/jdg/1497405628},
}

@article{song2018existence,
    AUTHOR = {Song, Antoine},
     TITLE = {Existence of infinitely many minimal hypersurfaces in closed
              manifolds},
   JOURNAL = {Ann. of Math. (2)},
  FJOURNAL = {Annals of Mathematics. Second Series},
    VOLUME = {197},
      YEAR = {2023},
    NUMBER = {3},
     PAGES = {859--895},
      ISSN = {0003-486X,1939-8980},
   MRCLASS = {53A10 (53C42)},
  MRNUMBER = {4564260},
MRREVIEWER = {Thomas\ Koerber},
       DOI = {10.4007/annals.2023.197.3.1},
       URL = {https://doi.org/10.4007/annals.2023.197.3.1},
}

@article{song2021generic,
    AUTHOR = {Song, Antoine and Zhou, Xin},
     TITLE = {Generic scarring for minimal hypersurfaces along stable
              hypersurfaces},
   JOURNAL = {Geom. Funct. Anal.},
  FJOURNAL = {Geometric and Functional Analysis},
    VOLUME = {31},
      YEAR = {2021},
    NUMBER = {4},
     PAGES = {948--980},
      ISSN = {1016-443X,1420-8970},
   MRCLASS = {53C42},
  MRNUMBER = {4317508},
MRREVIEWER = {Gabjin\ Yun},
       DOI = {10.1007/s00039-021-00571-7},
       URL = {https://doi.org/10.1007/s00039-021-00571-7},
}

@article {wang2022min,
    AUTHOR = {Wang, Tongrui},
     TITLE = {Min-max theory for {$G$}-invariant minimal hypersurfaces},
   JOURNAL = {J. Geom. Anal.},
  FJOURNAL = {Journal of Geometric Analysis},
    VOLUME = {32},
      YEAR = {2022},
    NUMBER = {9},
     PAGES = {Paper No. 233, 53},
      ISSN = {1050-6926,1559-002X},
   MRCLASS = {53A10 (53C42)},
  MRNUMBER = {4452896},
MRREVIEWER = {Peter\ McGrath},
       DOI = {10.1007/s12220-022-00966-4},
       URL = {https://doi.org/10.1007/s12220-022-00966-4},
}

@article{wang2023equivariant,
    AUTHOR = {Wang, Tongrui},
     TITLE = {Equivariant {M}orse index of min-max {$G$}-invariant minimal
              hypersurfaces},
   JOURNAL = {Math. Ann.},
  FJOURNAL = {Mathematische Annalen},
    VOLUME = {389},
      YEAR = {2024},
    NUMBER = {2},
     PAGES = {1599--1637},
      ISSN = {0025-5831,1432-1807},
   MRCLASS = {58E12 (49Q05)},
  MRNUMBER = {4745747},
MRREVIEWER = {Panayotis\ Vyridis},
       DOI = {10.1007/s00208-023-02681-z},
       URL = {https://doi.org/10.1007/s00208-023-02681-z},
}

@article{wang2026multiplicity,
	author = {Wang, Tongrui},
	journal = {arXiv preprint arXiv:2601.09884},
	title = {Multiplicity one for equivariant min-max theory in prescribed homology classes},
	year = {2026}}

@article {white1991space,
    AUTHOR = {White, Brian},
     TITLE = {The space of minimal submanifolds for varying {R}iemannian
              metrics},
   JOURNAL = {Indiana Univ. Math. J.},
  FJOURNAL = {Indiana University Mathematics Journal},
    VOLUME = {40},
      YEAR = {1991},
    NUMBER = {1},
     PAGES = {161--200},
      ISSN = {0022-2518,1943-5258},
   MRCLASS = {58D10 (53C42)},
  MRNUMBER = {1101226},
MRREVIEWER = {Jo\~ao\ Lucas Marques Barbosa},
       DOI = {10.1512/iumj.1991.40.40008},
       URL = {https://doi.org/10.1512/iumj.1991.40.40008},
}

@article{zhou2020multiplicity,
    AUTHOR = {Zhou, Xin},
     TITLE = {On the multiplicity one conjecture in min-max theory},
   JOURNAL = {Ann. of Math. (2)},
  FJOURNAL = {Annals of Mathematics. Second Series},
    VOLUME = {192},
      YEAR = {2020},
    NUMBER = {3},
     PAGES = {767--820},
      ISSN = {0003-486X,1939-8980},
   MRCLASS = {53C42 (49J35 49Q05 58E12)},
  MRNUMBER = {4172621},
MRREVIEWER = {Martin\ Man-Chun\ Li},
       DOI = {10.4007/annals.2020.192.3.3},
       URL = {https://doi.org/10.4007/annals.2020.192.3.3},
}

@article {zoll1903ueber,
    AUTHOR = {Zoll, Otto},
     TITLE = {Ueber {F}l\"achen mit {S}charen geschlossener geod\"atischer
              {L}inien},
   JOURNAL = {Math. Ann.},
  FJOURNAL = {Mathematische Annalen},
    VOLUME = {57},
      YEAR = {1903},
    NUMBER = {1},
     PAGES = {108--133},
      ISSN = {0025-5831,1432-1807},
   MRCLASS = {99-04},
  MRNUMBER = {1511201},
       DOI = {10.1007/BF01449019},
       URL = {https://doi.org/10.1007/BF01449019},
}
\end{document}